\pdfoutput=1

\documentclass{article}
\usepackage{graphicx} % Required for inserting images
\usepackage{amsfonts}
\usepackage{amsmath, amsthm, amssymb}
\usepackage[a4paper, top=2cm, bottom=2.5cm, left=2.5cm, right=2.5cm]{geometry}
\usepackage{tikz-cd}
\usepackage{tikz}
\usetikzlibrary{calc}
\usepackage{url}

\newtheorem{theorem}{Theorem}[section]
\newtheorem{definition}[theorem]{Definition}
\newtheorem{remark}{Remark}[theorem]
\newtheorem{corollary}[theorem]{Corollary}
\newtheorem{lemma}[theorem]{Lemma}

\newtheorem{example}{Example}[section]

\title{\Huge An Elementary Expository Study: \\ From Metric Spaces to Hilbert Spaces \\[1cm]}

\author{Ismail Gemaledin, Iusuf Gemaledin}
\date{October 2025}

\begin{document}

\maketitle

\newpage
\begin{abstract}
Metric spaces are a fundamental component of mathematics and have a paramount importance as a framework for measuring distance. They can be found in many different branches of mathematics, such as analysis and topology. This paper offers an elementary exposition of metric spaces and their associated topologies. We start by recalling the basic axioms through which we understand a metric and examine various examples. The induced topology is next discussed with emphasis on open and closed sets, continuity and limits. In addition, we compare equivalent metric spaces and illustrate how different metrics can generate but the same topological structure. The presentation is designed to be easy to follow and accessible to undergraduate students, by combining classical definitions with illustrative examples that allow a deeper understanding of the aforementioned concepts.
\end{abstract}

\newpage
\tableofcontents

\newpage
\section{Introduction}

The development of point-set topology owes much to the groundbreaking work of Georg Cantor in the late 19th century. Cantor introduced the theory of sets of points (Punktmannigfaltigkeitslehre) through a series of influential papers (1879–1883), establishing key concepts such as limit points, derived sets, dense and nowhere dense sets, and the famous Cantor “middle-third” set. His work provided mathematicians with powerful tools to analyze the structure of functions and spaces in a rigorous way. 

Building on Cantor’s ideas, Giuseppe Peano and Camille Jordan applied set-theoretic methods to geometry. Peano’s construction of a space-filling curve challenged intuitive notions of continuity and dimension, while Jordan formalized the Jordan curve theorem, proving that a simple closed curve divides the plane into an interior and exterior region. These contributions demonstrated the usefulness of set-theoretic thinking in both analysis and geometry. 

Around the turn of the 20th century, mathematicians such as Adolf Hurwitz and Arthur Schoenflies worked to systematize the study of point sets and topological invariants. Hurwitz proposed a framework for classifying point sets via continuous mappings, and Schoenflies expanded this work, rigorously developing planar topology and promoting Cantor’s methods. This era also witnessed the emergence of so-called “pathological” examples, such as curves with counterintuitive properties studied by Peano, Osgood, and Lebesgue. While initially surprising, these examples played a crucial role in shaping modern topology, showing the importance of precise definitions and providing deep insight into metric and topological structures. 

The notion of a metric space naturally arises from our intuitive understanding of distance. In mathematics, we often wish to measure how “close” or “far apart” two objects are, whether they are points on a line, functions, or more abstract entities. A metric provides a formal and rigorous way to capture this concept, allowing us to study convergence, continuity, and other foundational ideas systematically. Metric spaces form a cornerstone of modern mathematics, underpinning essential concepts in analysis, topology, and geometry, and finding applications in fields such as physics, computer science, and data analysis. 

The purpose of this expository paper is to present an accessible overview of metric spaces and their related notions. We draw on several standard sources, including Mendelson, Munkres, Kelley, Kasriel, and Schaum’s Outline, and integrate their material into a coherent exposition. The aim is to combine definitions, theorems, and illustrative examples in a way that is approachable for undergraduate students, while still capturing the rigor and depth of the subject. 

The paper is structured to guide the reader from foundational concepts to more advanced topics. It begins with metrics and pseudo-metrics, establishing the basic definitions and examples. Subsequent sections explore metric spaces, continuity, open balls and neighborhoods, and limits, providing the core concepts needed to understand topology in metric settings. Later, we examine open and closed sets, subspaces, and equivalent metrics, highlighting key results and examples that illustrate relationships between different structures. Finally, the article concludes with a discussion of infinite-dimensional Euclidean spaces, showing how these ideas extend to more abstract settings.

Beyond its theoretical elegance, the concept of a metric space has numerous applications — from defining distances between functions and probability distributions to measuring similarity in data science and machine learning. These modern uses highlight the enduring importance of metric structures in both pure and applied mathematics.

%                                                                   Metrics
%%%%%%%%%%%%%%%%%%%%%%%%%%%%%%%%%%%%%%%%%%%%%%%%%%%%%%%%%%%%%%%%%%%%%%%%%%%%%%%%%%%%%%%%%%%%%%%%%%%%%%%%%%%%%%%%%%%%%%%%%%%%%%%%%%%%%%%%%%%%%%%%%%
\newpage
\section{Metrics}

Before introducing metric spaces, it is essential to formalize the concept of "distance" between elements of a set. A metric, or distance function, provides such a formalization, assigning a non-negative real number to every pair of points and thereby quantifying their separation. This definition allows the notion of distance to be extended beyond familiar geometric contexts to more abstract sets, laying the foundation for the subsequent study of metric and topological structures. 

\begin{definition}[see \cite{Mendelson, Munkres, Kasriel, Armstrong, Ward, Viro}]
A metric is a function, we will write it as “d”, on a set X, usually known as the \emph{underlying set}. Also known as a distance function, the metric takes values in $\mathbb{R}$ to show the distance between two points, which we will usually name x and y. Therefore, “d” takes this form: 
\[d : X \times X \to \mathbb{R}\] 

\textit{If the function $d$ has the property that}
\[d(x, x) = 0, \quad \forall x \in X,\]
then it is called a \emph{pseudometric}. Most of the properties of said \emph{pseudometric} are the same as the one of the metric. If, in addition to the aforementioned axiom, the following four hold:

\begin{enumerate}
    \item $d(x, y) \geq 0$ $\forall$ $x, y \in X$;
    \item $d(x, y) = 0$ $\iff$ $x = y$ $\forall$ $x, y \in X$;
    \item $d(x, y) = d(y, x)$ $\forall$ $x, y \in X$;
    \item $d(x, z) \leq d(x, y) + d(y, z)$ $\forall$ $x, y, z \in X$;
\end{enumerate}
then the function $d$ is called a \emph{metric}. All four axioms must hold simultaneously.
\end{definition}

\begin{remark}
    The first axiom is self-explanatory; since all distances are naturally positive, metrics will always belong to the set $\mathbb{R}^+$.
\end{remark}

\begin{remark}
    The second axiom, when taking into consideration the latter x = y $\in$ X, becomes a retelling of the axiom of the pseudometric, that \[d(x, x) = 0, \quad \forall x \in X.\] Since the function "d" shows the distance of the two points it accepts as parameters, it is obvious that the distance between one point and itself is always equal to zero, irrespective of where it stands.
\end{remark}

\begin{remark}
    The third axiom expresses the symmetry of the said metric. The distance from point x to point y will always be equal to the distance between point y and point x.
\end{remark}

\begin{remark}
    The last of the axioms is called the \textbf{Triangle Inequality}, which states that any one of triangle's sides' length will always be smaller than the sum of the other two's lengths.
\end{remark}

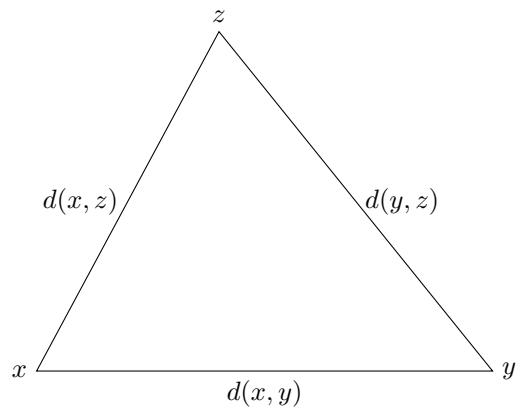
\begin{center}
\begin{tikzpicture}[scale=3]
    % Draw triangle
    \coordinate (x) at (0,0);
    \coordinate (y) at (2,0);
    \coordinate (z) at (0.8,1.5);

    \draw (x) -- (y) -- (z) -- cycle;

    % Label vertices
    \node[left]  at (x) {$x$};
    \node[right] at (y) {$y$};
    \node[above] at (z) {$z$};

    % Label sides
    \node[below] at ($(x)!0.5!(y)$) {$d(x, y)$};
    \node[right] at ($(y)!0.5!(z)$) {$d(y, z)$};
    \node[left]  at ($(x)!0.5!(z)$) {$d(x, z)$};
\end{tikzpicture}
\end{center}

\newpage

The concept is best illustrated by the following examples, especially using the renowned examples, such as the usual metric and the trivial metric. [see \cite{Schaum}]

\begin{example}[see \cite{Schaum}]
    The usual metric on the set $\mathbb{R}$ is the function:
    \[d(x, y) = |x - y|, \quad x, y \in \mathbb{R} \]

    %\newpage
    The usual metric on the $\mathbb{R}^2$ is very similar:
    \[d(a, b) = \sqrt{(x_1 - y_1)^2 \; + \; (x_2 - y_2)^2}, \quad x, y \in \mathbb{R},\]
    where a has the coordinates $x_1, x_2$ and where b has the coordinates $y_1, y_2$.
\end{example}

\begin{example}[see \cite{Kasriel}]
We can also have this metric in $\mathbb{R}^2$.
Let $k > 0$ and $d : \mathbf{R}^n \times \mathbf{R}^n \rightarrow \mathbf{R}$ be given by
\[
d(x, y) = k |x - y|, \quad \forall x, y\in \mathbf{R}^n.
\]
\end{example}

\begin{example}[see \cite{Schaum}]
    The trivial metric on a non-empty set X is the function "d", which looks like:
    \[d(x, y) = 
    \begin{cases}
        0, & x = y, \\
        1, & x \neq y.
    \end{cases}\]
    
\end{example}

\begin{example}[see \cite{Kasriel}]
    The \emph{Euclidean Metric} is the formula for the Euclidean distance for $\mathbb{R}^n$ and it takes the form:
    \[d(x, y) = |x - y| = [(x_1 - y_1)^2 + (x_2 - y_2)^2 + ... + (x_n - y_n)^2]^{1\over2}. \]
\end{example}

\begin{remark}[see \cite{Kasriel}]
    Let there be two points in the set $\mathbb{R}^n$, $x = (x_1, x_2, x_3, ..., x_n)$ and $y = (y_1, y_2, y_3, ..., y_n).$ For all $a \in P_n,$ we first have that the distance from x to y is always equal to or greater than the absolute value of the difference of $x_a$ and $y_a$:
    \[ d(x, y) \geqslant |x_a - y_a|. \]
    We also get the following inequality:
    \[ d(x, y) \leqslant \sqrt{n} \; max \{ \ | x_a - y_a | : a \in P_n \} \ \]
\end{remark}

\begin{example}[see \cite{Kasriel}]
    If we had a function $d : \mathbb{R}^2 \to \mathbb{R}^2$, it could also take the following form, which is also a metric:
    \[
    d(x, y) = \max \{ |x_a - y_a| : a = 1, 2 \}.
    \]
    In this case, a can take values one and two, and $x = (x_1, x_2), y = (y_1, y_2).$
\end{example}

\begin{example}[see \cite{Munkres}]
The following is known as the \emph{square metric} and is another well known example:
\[
d(x, y) = \max\{|x_1 - y_1|, ... , |x_n - y_n|\}.
\]
\end{example}

%de pus in Ch. 2
\begin{definition}[see \cite{Munkres}]
The \emph{norm} of x, where $x = (x_1, x_2, ..., x_n), \; x, x_1, x_2, ... , x_n \in \mathbb{R}^n$, is defined through the following equation:

\[
\|x\| = (x_1^2 + x_2^2 + ... + x_n^2)^{1/2}.
\]
\end{definition}

The concept of a metric, as defined above, allows us to measure distance in a precise and abstract way. However, the true power of this idea emerges when we consider the set on which the metric is defined. A metric does not exist in isolation—it provides structure to a set by introducing the notion of distance between its elements. This structure gives rise to what we call a \emph{metric space}. 

By studying metric spaces, we are able to explore how the properties of distance influence the behavior of points, sequences, and functions defined on a given set. In the next section, we formally introduce the definition of a metric space and illustrate how it provides the foundation for many of the key concepts in analysis and topology.

%                                                                   Metric Spaces
%%%%%%%%%%%%%%%%%%%%%%%%%%%%%%%%%%%%%%%%%%%%%%%%%%%%%%%%%%%%%%%%%%%%%%%%%%%%%%%%%%%%%%%%%%%%%%%%%%%%%%%%%%%%%%%%%%%%%%%%%%%%%%%%%%%%%%%%%%%%%%%%%%

\newpage
\section{Metric Spaces}

Now that we know what a metric is, we should be able to define and observe \emph{metric spaces.} They are comprised of two elements, a set and a distance function defined on said set, which has to satisfy all the aforementioned axioms. These \emph{metric spaces} represent one of most studied objects in analysis and point set topology and offer many concepts to be defined, such as convergence, continuity or compactness.

\begin{definition}[Metric Space][see \cite{Mendelson, Munkres, Kasriel}]

We will call a \emph{metric space} a pair (X, d), where X is a set and d is a metric on that set for which all of the four axioms defined in the second chapter of this paper.

\end{definition}

\begin{remark} [see \cite{Kasriel}]
    It has to be said that a single set can have many different metric spaces, since many different metrics can be defined on the same set.
\end{remark}

\begin{example}[see \cite{Kasriel}]
Assume the following function $d : \mathbf{R}^2 \times \mathbf{R}^2 \rightarrow \mathbf{R}$ given by
\[
d(x, y) = |x_1 - y_1| + |x_2 - y_2|
\]
where $x = (x_1, x_2)$ and $y = (y_1, y_2)$. This is a metric, and the pair $(\mathbb{R}^2, d)$ is the metric space.
\end{example}

\begin{remark}[see \cite{Mendelson}]
    Normally, a metric space should be written as the pair $(X, d : X \times X \to \mathbb{R})$, but in most cases the sets involved can be easily deduced from the context, so we shorten it to just $(X, d).$
\end{remark}

\begin{example}[see \cite{Mendelson, Kasriel}]
Imagine the function $d : \mathbf{R}^2 \times \mathbf{R}^2 \to \mathbf{R}$, where $d$ is:
\[
d(x, y) = \frac{|x - y|}{1 + |x - y|}.
\]
The function $d$ is a metric for the set $\mathbf{R}^2$, and the pair $(\mathbf{R}^2, d)$ is a metric space.
\end{example}

\begin{example}[see \cite{Kasriel}]
We can also have $x$ and $y$ take rational values.
    \[
    d(x, y) = |x - y|, \quad \forall x, y \in \mathbb{Q}.
    \]
    $(\mathbf{Q}, d)$ is a metric space, even though $d$ takes values in $\mathbb{R}$.
\end{example}

\begin{remark}[see \cite{Kasriel}]
    $d : \mathbf{Q} \times \mathbf{Q} \to \mathbf{R}$ is called the \emph{restriction} to $\mathbf{Q} \times \mathbf{Q}$ of the Euclidean metric for $\mathbb{R}$.
\end{remark}

\begin{definition}[see \cite{Munkres}]
Let \( A \) be a subset of \( X \), where the pair $(X, d)$ is a metric space. \( A \) is then called \textbf{bounded} if there is some number \( M \) such that, for every pair $(x, y)$ of any two points in \( A \), 
\[
d(a_1, a_2) \leq M.
\]
When the subset \( A \) is bounded and nonempty,
\[
\operatorname{diam} A = \sup \{ d(a_1, a_2) \mid a_1, a_2 \in A \}.
\]
is called the \textbf{diameter} of said subset \( A \).
\end{definition}

\begin{remark}[see \cite{Munkres}]
    Boundedness is not among the topological properties of a set, since it is dependent on a metric \( d \) on said set \( X \). For example, if we have $(X, d)$ as a metric space, the metric $\tilde{d}$ gives the topology of $X$, to which is bounded every single subset of the set.
\end{remark}

\begin{theorem}[see \cite{Munkres}]
Let $(X, d)$ be our metric space, where $d$ is our metric. We then have $\tilde{d} : X \times X \to \mathbb{R}, where:$

\(
    \tilde{d}(x, y) = \min \{ d(x, y), 1 \}.
\)

Both metric $d$ and metric $\tilde{d}$ induce the same \emph{topology}.

\end{theorem}

The metric \( \tilde{d} \) is called the \textbf{standard bounded metric} corresponding to \( d \).

\begin{proof}[Proof (see \cite{Munkres})]
    
We begin by verifying that $\tilde{d}$ satisfies the four axioms of a metric.  
The first three properties are immediate: $\tilde{d}(x, y) \ge 0$ for all $x, y$, $\tilde{d}(x, y) = 0$ holds if and only if $x = y$, and the metric always has the property of symmetry, for the distance between x and y is in this case equal to the distance from y to x.
Thus, the only nontrivial part to check is the triangle inequality.

We must show that
\[
\tilde{d}(x, z) \le \tilde{d}(x, y) + \tilde{d}(y, z), \forall x, y, z \in X.
\]

Consider first the case when at least one of $d(x, y)$ or $d(y, z)$ is greater than or equal to $1$.  
In that situation, the right-hand side of the inequality is at least $1$, while $\tilde{d}(x, z)$ is, by definition, never larger than $1$.  
Hence, the inequality holds automatically.

It remains to examine the case where both $d(x, y) < 1$ and $d(y, z) < 1$.  
For such points, we have $\tilde{d}(x, y) = d(x, y)$ and $\tilde{d}(y, z) = d(y, z)$.  
Using the triangle inequality for $d$, we obtain
\[
d(x, z) \le d(x, y) + d(y, z)
   = \tilde{d}(x, y) + \tilde{d}(y, z).
\]
Since $\tilde{d}(x, z) \le d(x, z)$ by the definition of $\tilde{d}$, the triangle inequality for $\tilde{d}$ follows.

Finally, we compare the topologies generated by $d$ and $\tilde{d}$.  
In any metric space, the collection of open balls with radius $\varepsilon < 1$ around each point forms a basis for the topology.  
Because for such $\varepsilon$ the balls defined by $d$ and those defined by $\tilde{d}$ coincide, both metrics produce the same open sets.  
Therefore, $d$ and $\tilde{d}$ generate identical topologies on $X$.
\end{proof}

\begin{lemma}[see \cite{Munkres}]
Let \( d \) and \( d' \) be two metrics defined on the same set \( X \), and let \( \mathcal{T} \) and \( \mathcal{T}' \) denote the topologies generated by each of them, respectively. Then \( \mathcal{T}' \) is finer than \( \mathcal{T} \) if and only if, for every \( x \in X \) and every \( \epsilon > 0 \), there exists some \( \delta > 0 \) such that
\[
B_{d'}(x, \delta) \subset B_d(x, \epsilon).
\]
\end{lemma}

\begin{proof}[see \cite{Munkres}]
Assume first that \( \mathcal{T}' \) is finer than \( \mathcal{T} \).  
Given any open ball \( B_d(x, \epsilon) \) belonging to the topology \( \mathcal{T} \), Lemma 13.3 guarantees that there exists a basic open set \( B' \) in \( \mathcal{T}' \) such that \( x \in B' \subset B_d(x, \epsilon) \).  
Since \( B' \) is open in \( \mathcal{T}' \), one can find a \( d' \)-ball \( B_{d'}(x, \delta) \) centered at \( x \) contained entirely in \( B' \), and therefore in \( B_d(x, \epsilon) \).

Conversely, suppose that for each \( x \in X \) and each \( \epsilon > 0 \), there exists a \( \delta > 0 \) satisfying \( B_{d'}(x, \delta) \subset B_d(x, \epsilon) \).  
Take any basic open set \( B_d(x, \epsilon) \) of \( \mathcal{T} \) that contains \( x \).  
By hypothesis, there is a \( \delta > 0 \) such that \( B_{d'}(x, \delta) \subset B_d(x, \epsilon) \), showing that every \( \mathcal{T} \)-open neighborhood of \( x \) contains a \( \mathcal{T}' \)-open neighborhood of \( x \).  
Hence, according to Lemma 13.3, \( \mathcal{T}' \) must be finer than \( \mathcal{T} \).
\end{proof}

\begin{example}
Consider the real line \( \mathbb{R} \) with two metrics:
\[
d(x, y) = |x - y| \quad \text{and} \quad d'(x, y) = \min\{|x - y|, 1\}.
\]
Here, \( d' \) "caps" distances at 1.  
For any point \( x \in \mathbb{R} \) and any \( \epsilon > 0 \), if we choose \( \delta = \min\{\epsilon, 1\} \), we have
\[
B_{d'}(x, \delta) \subset B_d(x, \epsilon).
\]
Thus, both induce the same topology — the usual topology on \( \mathbb{R} \).  
This illustrates how the lemma connects the inclusion of metric balls to the inclusion of the induced topologies.
\end{example}

\begin{definition}[see \cite{Munkres}]
Two metrics \( d_1 \) and \( d_2 \) on a set \( X \) are called \textbf{equivalent} if both generate the same topology on \( X \).
\end{definition}

\begin{example}
On \( \mathbb{R}^2 \), define
\[
d_1(x, y) = \sqrt{(x_1 - y_1)^2 + (x_2 - y_2)^2}, \quad 
d_2(x, y) = |x_1 - y_1| + |x_2 - y_2|.
\]
Although these two metrics measure distance differently (Euclidean vs. Manhattan), every open ball under one metric contains a smaller ball under the other.  
Hence, they define the same open sets — both correspond to the standard topology on \( \mathbb{R}^2 \).
\end{example}

\begin{theorem}[see \cite{Munkres}]
The topologies induced on \( \mathbb{R}^n \) by the Euclidean metric \( d \) and the square (or maximum) metric \( \rho \) coincide with the standard product topology on \( \mathbb{R}^n \).
\end{theorem}

\begin{proof}[see \cite{Munkres}]
Let \( x = (x_1, \ldots, x_n) \) and \( y = (y_1, \ldots, y_n) \) be points of \( \mathbb{R}^n \).  
A straightforward computation shows that
\[
\rho(x, y) \leq d(x, y) \leq \sqrt{n}\,\rho(x, y).
\]
From the first inequality, it follows that
\[
B_d(x, \epsilon) \subset B_\rho(x, \epsilon)
\]
for every \( x \) and \( \epsilon > 0 \), since \( d(x, y) < \epsilon \) implies \( \rho(x, y) < \epsilon \).  
Similarly, the second inequality gives
\[
B_\rho(x, \tfrac{\epsilon}{\sqrt{n}}) \subset B_d(x, \epsilon)
\]
for all \( x \) and \( \epsilon > 0 \).  
By the previous lemma, these inclusions imply that both metrics generate the same topology.

Next, we verify that this common topology coincides with the product topology.  
Let
\[
B = (a_1, b_1) \times \cdots \times (a_n, b_n)
\]
be a basic open set of the product topology, and choose \( x = (x_1, \ldots, x_n) \in B \).  
For each coordinate \( i \), select \( \epsilon_i > 0 \) such that
\[
(x_i - \epsilon_i,\, x_i + \epsilon_i) \subset (a_i, b_i),
\]
and let \( \epsilon = \min\{\epsilon_1, \ldots, \epsilon_n\} \).  
Then one can verify that \( B_\rho(x, \epsilon) \subset B \), so the \( \rho \)-topology is finer than the product topology.

Conversely, consider any \( \rho \)-ball \( B_\rho(x, \epsilon) \).  
For any point \( y \in B_\rho(x, \epsilon) \), we can define
\[
B = (x_1 - \epsilon, x_1 + \epsilon) \times \cdots \times (x_n - \epsilon, x_n + \epsilon),
\]
which is a product of open intervals and therefore a basic open set in the product topology.  
Clearly \( y \in B \subset B_\rho(x, \epsilon) \).  
Hence, the two topologies coincide.
\end{proof}

\begin{example}
In \( \mathbb{R}^2 \), consider the Euclidean metric
\[
d(x, y) = \sqrt{(x_1 - y_1)^2 + (x_2 - y_2)^2}
\]
and the square metric
\[
\rho(x, y) = \max\{|x_1 - y_1|, |x_2 - y_2|\}.
\]
The open \( \rho \)-ball centered at \( x \) with radius \( \epsilon \) is a square of side \( 2\epsilon \), while the open \( d \)-ball is a circle of radius \( \epsilon \).  
Both types of sets generate the same open neighborhoods of \( x \), so despite their different shapes, they define the same topology — the usual one on the plane.
\end{example}

\begin{remark}[see \cite{Munkres}]
Consider the space $\mathbb{R}^2$ with three common distance measures: the Euclidean metric, the square (or maximum) metric, and the taxicab (Manhattan) metric. The shapes of the corresponding unit balls are different: circular for the Euclidean metric, square for the maximum metric, and diamond-shaped for the taxicab metric. Despite these geometric differences, all three metrics generate the same topology on $\mathbb{R}^2$. This demonstrates that different metrics can lead to the same notion of "closeness" or open sets, even though distances may be measured differently.
\end{remark}

\begin{example}
In $\mathbb{R}^2$, let $d_2$ denote the Euclidean metric, $d_\infty$ the maximum metric, and $d_1$ the taxicab metric. The unit balls around the origin are:
\[
B_{d_2}(0,1) = \{(x,y) \in \mathbb{R}^2 : x^2 + y^2 < 1\}, \quad
B_{d_\infty}(0,1) = \{(x,y) : \max(|x|,|y|) < 1\}, 
\]
\[
B_{d_1}(0,1) = \{(x,y) : |x| + |y| < 1\}.
\]
Even though their shapes differ, the sets of open neighborhoods generated by each metric define the same topology on $\mathbb{R}^2$.
\end{example}

\begin{theorem}[see \cite{Munkres}]
Let $\tilde{d}(a, b) = \min\{|a - b|, 1\}$ be the bounded metric on $\mathbb{R}$. For points $x = (x_1, x_2, \dots)$ and $y = (y_1, y_2, \dots)$ in the infinite product $\mathbb{R}^\omega$, define
\[
D(x, y) = \sup_{i \in \mathbb{Z}_+} \frac{\tilde{d}(x_i, y_i)}{i}.
\]
Then $D$ is a metric on $\mathbb{R}^\omega$ and the topology it induces coincides with the usual product topology.
\end{theorem}

\begin{proof}[see \cite{Munkres}]
All metric axioms are straightforward except the triangle inequality. For any $x, y, z \in \mathbb{R}^\omega$ and each $i$:
\[
\frac{\tilde{d}(x_i, z_i)}{i} \le \frac{\tilde{d}(x_i, y_i)}{i} + \frac{\tilde{d}(y_i, z_i)}{i} \le D(x, y) + D(y, z).
\]
Taking the supremum over all $i$ gives
\[
D(x, z) \le D(x, y) + D(y, z),
\]
so the triangle inequality holds.

To see that $D$ generates the product topology, take any open set $U$ in the metric topology and a point $x \in U$. Choose $\epsilon > 0$ so that the $D$-ball $B_D(x, \epsilon) \subset U$. Pick $N$ large enough such that $1/N < \epsilon$, and define the product basis element
\[
V = \prod_{i=1}^N (x_i - \epsilon, x_i + \epsilon) \times \prod_{i > N} \mathbb{R}.
\]
For any $y \in V$,
\[
\frac{\tilde{d}(x_i, y_i)}{i} \le \frac{1}{N} < \epsilon \quad \text{for } i > N,
\]
so $y \in B_D(x, \epsilon)$. Hence, $V \subset B_D(x, \epsilon) \subset U$.

Conversely, consider a product basis element
\[
U = \prod_{i \in \mathbb{Z}_+} U_i
\]
where only finitely many $U_i \neq \mathbb{R}$, say $i = \alpha_1, \dots, \alpha_n$. Let $x \in U$, and choose intervals $(x_i - \epsilon_i, x_i + \epsilon_i) \subset U_i$ with $\epsilon_i \le 1$. Define 
\[
\epsilon = \min \{\epsilon_i/i : i = \alpha_1, \dots, \alpha_n \}.
\]
Then the $D$-ball $B_D(x, \epsilon)$ satisfies $x \in B_D(x, \epsilon) \subset U$, verifying that $D$ indeed induces the product topology.
\end{proof}

Metric spaces provide a versatile framework for capturing the notion of distance. By defining different metrics on the same set, we can either preserve the underlying topology or create new ones, allowing us to model a variety of geometric and analytic scenarios. 

In particular, metrics let us generalize familiar Euclidean concepts like convergence and continuity to much more abstract settings. In the following sections, we will see how these ideas allow us to reason about limits, continuous functions, and other fundamental properties in arbitrary metric spaces.

%pana aici am rescris totul de mana
%%%%%%%%%%%%%%%%%%%%%%%%%%%%%%%%%%%%%%%%%%%%%%%%%%%%%%%%%%%%%%%%%%%%%%%%%%%%%%%%%%%%%%%%%%%%%%%%%%%%%%%%%%%%%%%%%%%%%%%%%%%%%%%%%%%%%%%%%%%%%%%%%%%%%%%%%%%%%%%%%%%%%%%%

%                                                               Pseudometric Spaces
%%%%%%%%%%%%%%%%%%%%%%%%%%%%%%%%%%%%%%%%%%%%%%%%%%%%%%%%%%%%%%%%%%%%%%%%%%%%%%%%%%%%%%%%%%%%%%%%%%%%%%%%%%%%%%%%%%%%%%%%%%%%%%%%%%%%%%%%%%%%%%%%%%

\newpage
\section{Pseudo-Metric Spaces}

In classical metric spaces, the distance between two distinct points is always positive. However, there are contexts where it is natural to allow different points to have zero distance. This leads to the concept of a \emph{pseudo-metric space}, which generalizes metric spaces by relaxing the condition that zero distance implies equality.

\begin{definition}[Pseudo-Metric, see \cite{KelleyGeneralTopology}]
Let \(X\) be a non-empty set. A function
\[
d: X \times X \to \mathbb{R}_{\ge 0}
\]
is called a \emph{pseudo-metric} if, for all \(x, y, z \in X\):
\begin{enumerate}
    \item \(d(x,x) = 0\),
    \item \(d(x,y) = d(y,x)\) (symmetry),
    \item \(d(x,z) \le d(x,y) + d(y,z)\) (triangle inequality).
\end{enumerate}
\end{definition}

\begin{remark}
In a pseudo-metric, two different points may have zero distance. Hence, every metric is a pseudo-metric, but pseudo-metrics allow more flexibility.
\end{remark}

\begin{definition}[Pseudo-Metric Space]
A set \(X\) equipped with a pseudo-metric \(d\) is called a \emph{pseudo-metric space}, denoted \((X,d)\).
\end{definition}

\begin{example}[Trivial Zero Pseudo-Metric]
For any set \(X\), define \(d(x,y) = 0\) for all \(x,y \in X\). Then \((X,d)\) is a pseudo-metric space. If \(|X|>1\), it is not a metric space because distinct points have zero distance.
\end{example}

\begin{example}[Single-Coordinate Pseudo-Metric]
Consider \(X = \mathbb{R}^3\) and define
\[
d((x_1,x_2,x_3),(y_1,y_2,y_3)) = |x_1 - y_1|.
\]
This defines a pseudo-metric because it satisfies symmetry, non-negativity, and the triangle inequality, yet it ignores the second and third coordinates entirely.
\end{example}

\begin{example}[String Comparison Pseudo-Metric]
Let \(X\) be the set of all finite strings over a fixed alphabet. Define \(d(s_1, s_2) = 0\) if the first character of \(s_1\) equals that of \(s_2\), and \(1\) otherwise. This function is a pseudo-metric: it satisfies symmetry, triangle inequality, and \(d(s,s)=0\), but distinct strings may have zero distance if they start with the same character.
\end{example}

\begin{remark}
Many standard notions from metric spaces extend naturally to pseudo-metrics, such as convergence of sequences and continuity of functions. The main difference is that distinct points can be "indistinguishable" with respect to the pseudo-metric.
\end{remark}

In pseudo-metric spaces, we often need to measure how far a point is from a subset or to analyze collections of points that are "close" in the pseudo-metric sense. Unlike metric spaces, two distinct points may have zero distance, which leads to interesting phenomena and constructions.

\begin{definition}[Distance to a subset]
Let \( (X,d) \) be a pseudo-metric space and \( A \subset X \). For any \( x \in X \), define the \emph{distance from \(x\) to \(A\)} as
\[
D(A,x) = \inf\{ d(x,a) : a \in A \}.
\]
\end{definition}

\begin{example}
Let \(X = \mathbb{R}^2\) with pseudo-metric \(d((x_1,x_2),(y_1,y_2)) = |x_1 - y_1|\) and let \(A = \{(0,y) : y \in \mathbb{R}\}\). Then for any point \((x_1,x_2)\), \(D(A,(x_1,x_2)) = |x_1|\), independent of \(x_2\).
\end{example}

\begin{remark}
In pseudo-metric spaces, a point \( x \) with \( D(A,x) = 0 \) might not belong to \( A \), unlike in strict metric spaces.
\end{remark}

\begin{theorem}[Countability axioms]
Pseudo-metric spaces always satisfy the first axiom of countability. They satisfy the second axiom of countability if and only if they are separable.
\end{theorem}

\begin{theorem}[Nets and convergence]
A net \(\{x_\alpha\}\) in a pseudo-metric space \((X,d)\) converges to \( x \in X \) if and only if \( d(x_\alpha, x) \to 0 \).
\end{theorem}

\begin{definition}[Diameter of a set]
For any \( A \subset X \), define the \emph{diameter} by
\[
\operatorname{diam} A = \sup \{ d(x,y) : x,y \in A \}.
\]
If the supremum does not exist, the diameter is considered infinite.
\end{definition}

\begin{remark}
Finite diameter is not a topological property; it may change under homeomorphisms or pseudo-metric transformations.
\end{remark}

\begin{theorem}[Bounded pseudo-metrics]
Let \(e(x,y) = \min\{1, d(x,y)\}\). Then \( (X,e) \) is a pseudo-metric space with diameter at most 1. The topology induced by \(e\) coincides with the topology induced by \(d\).
\end{theorem}

\begin{remark}
This allows us to assume pseudo-metric spaces have diameter at most one without loss of generality when studying topological properties.
\end{remark}

\begin{theorem}[Countable products]
Let \(\{(X_n,d_n)\}\) be a sequence of pseudo-metric spaces with diameters \(\le 1\). Define a pseudo-metric on the product \(\prod_n X_n\) by
\[
d(x,y) = \sum_{n=0}^{\infty} 2^{-n} d_n(x_n,y_n).
\]
Then \(d\) defines a pseudo-metric whose induced topology agrees with the product topology.
\end{theorem}

\begin{definition}[Isometry]
Let \((X,d)\) and \((Y,e)\) be pseudo-metric spaces. A map \( f: X \to Y \) is an \emph{isometry} if
\[
d(x,y) = e(f(x),f(y)) \quad \text{for all } x,y \in X.
\]
\end{definition}

\begin{remark}
Unlike in metric spaces, isometries in pseudo-metric spaces need not be injective: points at zero distance may collapse. Topological and metric invariants are preserved under isometries.
\end{remark}

\begin{definition}[Metric quotient]
For a pseudo-metric space \((X,d)\), define an equivalence relation
\[
x \sim y \iff d(x,y) = 0.
\]
Let \(\mathfrak{D}\) denote the set of equivalence classes. Define a distance \(D\) on \(\mathfrak{D}\) by
\[
D([x],[y]) = \inf \{ d(a,b) : a \in [x], b \in [y] \}.
\]
Then \((\mathfrak{D},D)\) is a metric space, called the \emph{metric quotient} of \((X,d)\).
\end{definition}

\begin{remark}
Every pseudo-metric space is “almost” a metric space: the metric quotient identifies points at zero distance, yielding a genuine metric space suitable for standard metric-space analysis.
\end{remark}

In summary, pseudo-metric spaces generalize the familiar concept of metric spaces by allowing distinct points to have zero distance, while still retaining many useful properties such as continuity of the distance function, normality, and first-countability. Key constructions, including bounded pseudo-metrics, countable products, and metric quotients, enable us to study pseudo-metric spaces using the same topological tools as for metric spaces. The notion of isometries and metric quotients illustrates that every pseudo-metric space can be “converted” into a genuine metric space, revealing a deep connection between these two classes of spaces. Understanding these foundational aspects sets the stage for exploring further topological and functional-analytic structures built upon pseudo-metrics.

%                                                        Continuity and Convergence
%%%%%%%%%%%%%%%%%%%%%%%%%%%%%%%%%%%%%%%%%%%%%%%%%%%%%%%%%%%%%%%%%%%%%%%%%%%%%%%%%%%%%%%%%%%%%%%%%%%%%%%%%%%%%%%%%%%%%%%%%%%%%%%%%%%%%%%%%%%%%%%%%%

\newpage
\section{Continuity and Convergence}

The concepts of convergence and continuity lie at the heart of analysis and topology. They describe how mathematical objects behave under limiting processes and transformations.  
Convergence captures the idea that a sequence of points can approach a specific value, while continuity expresses the idea that small changes in the input of a function produce only small changes in its output.  
Together, these notions allow us to understand stability and approximation — ideas that appear throughout mathematics, from calculus and differential equations to functional analysis and topology.

In this chapter, we begin by formalizing the notion of convergence in a general metric space, extending the familiar definition from the real numbers. We then introduce the concept of continuity for functions between metric spaces and explore how it can be described both in terms of the $\varepsilon$–$\delta$ definition and through neighborhoods or sequences.  
Finally, we study key properties of continuous functions, including the behavior of constant and identity maps, equivalence under different metrics, and the preservation of continuity under composition. These results build the foundation for the study of topological structures that follows in later chapters.

\begin{definition}[Convergence of sequences, see \cite{Schaum}]
Let $(X, d)$ be a metric space. A sequence $\langle a_1, a_2, \dots \rangle$ of elements in $X$ is said to \textbf{converge} to a point $b \in X$ if, for every $\epsilon > 0$, there exists a positive integer $n_0$ such that
\[
n > n_0 \implies d(a_n, b) < \epsilon.
\]
In this case, we write $\lim_{n \to \infty} a_n = b$.
\end{definition}

\begin{remark}
Convergence of sequences can be viewed as a particular instance of convergence of \emph{nets} in topology. Every sequence is a net whose index set is the set of natural numbers $\mathbb{N}$.
\end{remark}

\begin{example}[see \cite{Schaum}]
In $(\mathbb{R}, |\cdot|)$, the sequence $a_n = \frac{1}{n}$ converges to $0$, since for every $\epsilon > 0$, choosing $n_0 > 1/\epsilon$ ensures that $|a_n - 0| < \epsilon$ whenever $n > n_0$.
\end{example}

\begin{definition}[Continuity on $\mathbb{R}$, see \cite{Mendelson}]
Let $f: \mathbb{R} \to \mathbb{R}$. The function $f$ is said to be \textbf{continuous at a point} $a \in \mathbb{R}$ if for every $\epsilon > 0$ there exists a $\delta > 0$ such that
\[
|x - a| < \delta \implies |f(x) - f(a)| < \epsilon.
\]
We say that $f$ is \textbf{continuous} on $\mathbb{R}$ if it is continuous at every $a \in \mathbb{R}$.
\end{definition}

\begin{example}[see \cite{Mendelson}]
The identity function $f(x) = x$ is continuous at all points $a \in \mathbb{R}$, since for any $\epsilon > 0$, choosing $\delta = \epsilon$ gives $|x - a| < \delta \implies |f(x) - f(a)| = |x - a| < \epsilon$.
\end{example}

\begin{definition}[Continuity in metric spaces, see \cite{Mendelson}]
Let $(X, d)$ and $(Y, d')$ be metric spaces. A function $f: X \to Y$ is \textbf{continuous at a point} $a \in X$ if, for every $\epsilon > 0$, there exists $\delta > 0$ such that
\[
d(x, a) < \delta \implies d'(f(x), f(a)) < \epsilon.
\]
The function $f$ is \textbf{continuous} on $X$ if it satisfies this condition at each $a \in X$.
\end{definition}

\begin{remark}
An equivalent characterization is that $f: X \to Y$ is continuous if and only if the preimage of every open set in $Y$ is open in $X$. This equivalence links the metric definition of continuity with the general topological notion.
\end{remark}

\begin{example}[Linear functions, see \cite{Kasriel}]
Let $d_n$ and $d$ denote the Euclidean metrics on $\mathbb{R}^n$ and $\mathbb{R}$, respectively. Fix $a \in \mathbb{R}^n$ and define
\[
f(x) = a \cdot x.
\]
Then $f$ is continuous from $(\mathbb{R}^n, d_n)$ into $(\mathbb{R}, d)$.

\emph{Proof sketch:}  
For $x_0 \in \mathbb{R}^n$ and $\varepsilon > 0$, if $|a| = 0$, the statement is trivial. Otherwise, let $\delta = \varepsilon / |a|$. Then
\[
d_n(x, x_0) < \delta \implies |f(x) - f(x_0)| = |a \cdot (x - x_0)| \le |a| \, \delta = \varepsilon.
\]
Hence $f$ is continuous.
\end{example}

\begin{example}[Quadratic function, see \cite{Mendelson}]
Let $f: \mathbb{R} \to \mathbb{R}$ be defined by $f(x) = x^2$.  
For $a \in \mathbb{R}$ and $\epsilon > 0$, choose $\delta = \min\{1, \epsilon / (2|a| + 1)\}$.  
Then
\[
|x - a| < \delta \implies |x^2 - a^2| = |x - a||x + a| < \epsilon,
\]
showing that $f$ is continuous at $a$.
\end{example}

\begin{theorem}[Constant functions are continuous, see \cite{Mendelson}]
Let $(X, d)$ and $(Y, d')$ be metric spaces. If $f: X \to Y$ is a constant function, then $f$ is continuous on $X$.
\end{theorem}

\begin{proof}
Fix any $a \in X$ and let $\epsilon > 0$.  
Since $f$ is constant, there exists some $c \in Y$ such that $f(x) = c$ for all $x \in X$.  
Thus, for every $x \in X$, we have
\[
d'(f(x), f(a)) = d'(c, c) = 0 < \epsilon.
\]
This inequality holds for any $\delta > 0$, so we may take, for instance, $\delta = 1$.  
Hence $f$ is continuous at $a$, and therefore continuous on $X$.
\end{proof}

\begin{example}
Consider $f: \mathbb{R} \to \mathbb{R}$ defined by $f(x) = 7$.  
No matter how $x$ varies, $|f(x) - f(a)| = 0 < \epsilon$ for all $\epsilon > 0$, showing that $f$ is continuous everywhere.
\end{example}

\begin{theorem}[The identity map is continuous, see \cite{Mendelson}]
Let $(X, d)$ be a metric space. The identity function $i: X \to X$, defined by $i(x) = x$, is continuous.
\end{theorem}

\begin{proof}
Let $a \in X$ and $\epsilon > 0$ be given.  
Choose $\delta = \epsilon$.  
Then for any $x \in X$ satisfying $d(x, a) < \delta$, we have
\[
d(i(x), i(a)) = d(x, a) < \epsilon.
\]
Thus $i$ is continuous at $a$, and consequently continuous on $X$.
\end{proof}

\begin{theorem}[Identity map between equivalent metrics on $\mathbb{R}^n$, see \cite{Mendelson}]
Let $i: \mathbb{R}^n \to \mathbb{R}^n$ be the identity function.  
Denote by $d$ the maximum (supremum) metric and by $d'$ the Euclidean metric on $\mathbb{R}^n$.  
Then both mappings
\[
i: (\mathbb{R}^n, d) \to (\mathbb{R}^n, d') 
\quad \text{and} \quad 
i: (\mathbb{R}^n, d') \to (\mathbb{R}^n, d)
\]
are continuous.
\end{theorem}

\begin{proof}
Let $a = (a_1, \dots, a_n) \in \mathbb{R}^n$ and $\epsilon > 0$.  

\textbf{(1)} For $i: (\mathbb{R}^n, d) \to (\mathbb{R}^n, d')$:  
If $d(x, a) < \delta$, then each $|x_k - a_k| < \delta$ for $k = 1, \dots, n$.  
Hence
\[
d'(x, a) = \sqrt{\sum_{k=1}^n (x_k - a_k)^2} 
< \sqrt{n} \, \delta.
\]
By choosing $\delta = \epsilon / \sqrt{n}$, we obtain $d'(x, a) < \epsilon$.

\textbf{(2)} For $i: (\mathbb{R}^n, d') \to (\mathbb{R}^n, d)$:  
If $d'(x, a) < \delta$, then in particular $|x_k - a_k| \le d'(x, a) < \delta$ for all $k$.  
Thus
\[
d(x, a) = \max_k |x_k - a_k| < \delta.
\]
Taking $\delta = \epsilon$ ensures $d(x, a) < \epsilon$, proving continuity.

Therefore, $i$ is continuous for both pairs of metrics.
\end{proof}

\begin{example}
For $n = 2$, let $d(x, y) = \max\{|x_1 - y_1|, |x_2 - y_2|\}$ and 
$d'(x, y) = \sqrt{(x_1 - y_1)^2 + (x_2 - y_2)^2}$.  
Then the identity function $i(x_1, x_2) = (x_1, x_2)$ satisfies  
$d'(x, a) \le \sqrt{2} \, d(x, a)$ and $d(x, a) \le d'(x, a)$, confirming that $i$ is continuous both ways.
\end{example}

\begin{theorem}[Composition of continuous mappings, see \cite{Mendelson}]
Let $(X, d)$, $(Y, d')$, and $(Z, d'')$ be metric spaces.  
Suppose that $f: X \to Y$ is continuous at a point $a \in X$, and that $g: Y \to Z$ is continuous at the point $f(a) \in Y$.  
Then the composition $g \circ f: X \to Z$ is continuous at $a$.
\end{theorem}

\begin{proof}
Fix $\epsilon > 0$.  
Since $g$ is continuous at $f(a)$, there exists $\eta > 0$ such that
\[
d'(y, f(a)) < \eta \implies d''(g(y), g(f(a))) < \epsilon.
\]
Because $f$ is continuous at $a$, there exists $\delta > 0$ for which
\[
d(x, a) < \delta \implies d'(f(x), f(a)) < \eta.
\]
Combining these two implications, whenever $d(x, a) < \delta$, we have
\[
d''(g(f(x)), g(f(a))) < \epsilon.
\]
Hence $g \circ f$ is continuous at $a$.
\end{proof}

\begin{corollary}[see \cite{Mendelson}]
If $f: X \to Y$ and $g: Y \to Z$ are continuous on their entire domains, then the composite function $g \circ f: X \to Z$ is also continuous on $X$.
\end{corollary}

\begin{example}
Let $f: \mathbb{R} \to \mathbb{R}$ be defined by $f(x) = x^2$ and $g: \mathbb{R} \to \mathbb{R}$ by $g(y) = \sin(y)$.  
Both functions are continuous on $\mathbb{R}$.  
Their composition is
\[
(g \circ f)(x) = \sin(x^2),
\]
which is continuous on $\mathbb{R}$ as a result of the theorem above.
\end{example}

\begin{remark}
This result can be viewed as a generalization of the familiar fact from real analysis that the composition of continuous real functions is continuous.  
The argument relies only on the $\varepsilon$–$\delta$ definition, so it applies to any metric (or even pseudometric) spaces.
\end{remark}

\begin{remark}
An important consequence is that complicated continuous maps can be built step by step from simpler ones — such as linear, polynomial, or trigonometric functions — while preserving continuity throughout the construction.
\end{remark}

In this chapter, we generalized the familiar ideas of convergence and continuity from the real line to arbitrary metric spaces. We saw that convergence of sequences provides a natural way to describe continuity: a function is continuous if and only if it preserves limits.  
We also observed that simple functions — such as constant and identity mappings — are continuous in any metric space, and that the composition of continuous functions remains continuous.  

These properties highlight the robustness of the concept of continuity: it behaves predictably and harmoniously under basic operations, forming a bridge between analysis and topology.  
In the chapters that follow, this foundation will allow us to reinterpret continuity, openness, and convergence purely in topological terms, without explicit reference to a metric.

% ---------- Open Balls and Neighborhoods (cleaned) ----------
\newpage
\section{Open Balls and Neighborhoods}

When studying the behavior of functions between metric spaces, one of the central ideas is continuity. 
Intuitively, a function $f : X \to Y$ is continuous at a point $a \in X$ if small changes in $x$ near $a$ lead to small changes in $f(x)$ near $f(a)$. 
To express this idea rigorously, we need a precise way to describe what it means for points in $X$ to be "close" to one another.

In a metric space $(X, d)$, the notion of distance is already given by the metric $d$. 
Thus, for any point $a \in X$, it is natural to consider all points $x$ whose distance from $a$ is less than some positive number $\delta$. 
This collection of points represents the region around $a$ that lies within distance $\delta$, and it plays a fundamental role in the definitions of continuity, limits, and openness.

We formalize this idea as follows:

\begin{definition}[see \cite{Mendelson, negrescu}]
Let $(X, d)$ be a metric space.  
For a point $a \in X$ and a real number $\delta > 0$, we define the \textbf{open ball} centered at $a$ with radius $\delta$ as the set
\[
B(a; \delta) = \{\, x \in X \mid d(a, x) < \delta \,\}.
\]
In other words, $B(a; \delta)$ contains all points of $X$ whose distance from $a$ is strictly less than $\delta$.
\end{definition}

\begin{remark}
Geometrically, the open ball \(B(a;\delta)\) represents all points of \(X\) lying strictly within distance \(\delta\) from \(a\). Its shape depends on the chosen metric.
\end{remark}

Thus, \(x \in B(a;\delta)\) if and only if \(d(x,a) < \delta\). Similarly, if \((Y, d')\) is another metric space and \(f: X \to Y\), then \(y \in B(f(a);\epsilon)\) if and only if \(d'(y,f(a)) < \epsilon\).

\begin{theorem}[see \cite{Mendelson}]
A function \(f:(X, d) \to (Y, d')\) is continuous at a point \(a \in X\) if and only if for every \(\epsilon > 0\) there exists \(\delta > 0\) such that
\[
f\big(B(a;\delta)\big) \subseteq B\big(f(a);\epsilon\big).
\]
\end{theorem}

\begin{remark}
This reformulation shows that continuity means: points sufficiently close to \(a\) in \(X\) are mapped by \(f\) to points sufficiently close to \(f(a)\) in \(Y\).
\end{remark}

\begin{theorem}[see \cite{Mendelson}]
A function \(f:(X, d) \to (Y, d')\) is continuous at \(a \in X\) if and only if, for every \(\epsilon > 0\), there exists \(\delta > 0\) such that
\[
B(a;\delta) \subseteq f^{-1}\big(B(f(a);\epsilon)\big).
\]
\end{theorem}

\begin{definition}[see \cite{Mendelson}]
Let $(X, d)$ be a metric space and let $a \in X$.  
A subset $N \subset X$ is called a \textbf{neighborhood} of the point $a$ if there exists a real number $\delta > 0$ such that the open ball centered at $a$ with radius $\delta$ is contained in $N$; that is,
\[
B(a; \delta) \subset N.
\]
The collection of all neighborhoods of \(a\) is denoted by \(\mathfrak{R}_a\) and is called the \textbf{complete system of neighborhoods} of \(a\).
\end{definition}

\begin{remark}
Each open ball around \(a\) is a neighborhood of \(a\), but not every neighborhood must itself be an open ball—it only needs to contain one.
\end{remark}

\begin{lemma}[adapted from \cite{Mendelson}]
Let $(X, d)$ be a metric space and let $a \in X$.  
For every $\delta > 0$, the open ball $B(a; \delta)$ serves as a neighborhood for each point it contains.
\end{lemma}

\begin{proof}
Let \(b \in B(a;\delta)\). Since \(d(a,b) < \delta\), choose \(\eta = \delta - d(a,b) > 0\). If \(x \in B(b;\eta)\), then by the triangle inequality
\[
d(a,x) \le d(a,b) + d(b,x) < d(a,b) + \eta = \delta,
\]
so \(x \in B(a;\delta)\). Hence \(B(b;\eta) \subseteq B(a;\delta)\).
\end{proof}

\begin{remark}
This lemma shows that open balls are ``open'' in the intuitive sense: around every point inside them we can fit a smaller ball still contained within.
\end{remark}

\begin{theorem}[Neighborhood characterization of continuity, adapted from \cite{Mendelson}]
Let $f: (X, d) \to (Y, d')$ be a function between metric spaces.  
Then $f$ is continuous at a point $a \in X$ if and only if, for every neighborhood $M$ of $f(a)$, there exists a neighborhood $N$ of $a$ such that  
\[
f(N) \subset M,
\]
which is equivalent to saying that  
\[
N \subset f^{-1}(M).
\]
\end{theorem}

\begin{proof}
($\Rightarrow$) Suppose \(f\) is continuous at \(a\). For any neighborhood \(M\) of \(f(a)\), choose \(\epsilon > 0\) with \(B(f(a);\epsilon) \subseteq M\). By continuity there exists \(\delta > 0\) such that \(f(B(a;\delta)) \subseteq B(f(a);\epsilon)\). Set \(N = B(a;\delta)\).

($\Leftarrow$) Conversely, let \(M = B(f(a);\epsilon)\). By hypothesis there is a neighborhood \(N\) of \(a\) with \(f(N) \subseteq M\). Since \(N\) contains some ball \(B(a;\delta)\), we have \(f(B(a;\delta)) \subseteq M\), proving continuity.
\end{proof}

\begin{theorem}[Equivalent neighborhood formulation, adapted from \cite{Mendelson}]
Let $f: (X, d) \to (Y, d')$ be a function between metric spaces.  
Then $f$ is continuous at a point $a \in X$ if and only if, for every neighborhood $M$ of $f(a)$ in $Y$, the preimage $f^{-1}(M)$ is a neighborhood of $a$ in $X$.
\end{theorem}

\begin{theorem}[Neighborhood system properties, see \cite{Mendelson}]
Let \((X, d)\) be a metric space. The collection of neighborhoods satisfies:
\begin{enumerate}
  \item[$N_1$] For every point $a \in X$, there exists at least one neighborhood that contains $a$.
  \item[$N_2$] If \(N\) is a neighborhood of \(a\), then \(a \in N\).
  \item[$N_3$] If \(N\) is a neighborhood of \(a\) and \(N' \supseteq N\), then \(N'\) is also a neighborhood of \(a\).
  \item[$N_4$] The intersection of two neighborhoods of \(a\) is a neighborhood of \(a\).
  \item[$N_5$] For each neighborhood \(N\) of \(a\), there exists a neighborhood \(O \subseteq N\) of \(a\) that is also a neighborhood of each of its points.
\end{enumerate}
\end{theorem}

\begin{proof}
\(N_1\): \(X\) is a neighborhood of every point.  
\(N_2\): Trivial.  
\(N_3\): If \(B(a;\delta) \subseteq N \subseteq N'\), then \(B(a;\delta) \subseteq N'\).  
\(N_4\): If \(B(a;\delta_1)\subseteq N\) and \(B(a;\delta_2)\subseteq M\), then \(B(a;\min\{\delta_1,\delta_2\})\subseteq N\cap M\).  
\(N_5\): Direct from Lemma above.
\end{proof}

\begin{definition}[Basis for neighborhood systems, see \cite{Mendelson}]
Let \(a \in X\). A collection \(\mathfrak{B}_a\) of neighborhoods of \(a\) is a \textbf{basis for the neighborhood system at \(a\)} if every neighborhood \(N\) of \(a\) contains some \(B \in \mathfrak{B}_a\).
\end{definition}

\begin{example}
On the real line \(\mathbb{R}\), a basis for the neighborhood system at \(a\) is the set of open intervals \((a-\delta,a+\delta)\) with \(\delta>0\).
\end{example}

\newpage
\section{Limits}

Before extending the notion of limit to general metric spaces, let us first recall the classical definition for sequences of real numbers.

\begin{definition}[see \cite{Mendelson, negrescu}]
Let $(a_n)_{n \ge 1}$ be a sequence of real numbers. We say that a real number $a$ is the \textbf{limit} of the sequence if for every $\varepsilon > 0$, there exists an integer $N > 0$ such that 
\[
n > N \quad \implies \quad |a_n - a| < \varepsilon.
\]
In this case, we also say that the sequence \((a_n)\) \textbf{converges} to $a$ and write
\[
\lim_{n \to \infty} a_n = a.
\]
\end{definition}

\begin{remark}
Intuitively, $\varepsilon$ represents an arbitrarily small margin of error, while $N$ indicates a position far enough along the sequence so that all subsequent terms lie within this margin of the limit.
\end{remark}

This idea can be generalized to sequences in any metric space $(X, d)$. Let $(a_n)_{n \ge 1}$ be a sequence of points in $X$, and let $a \in X$. Consider the real sequence of distances
\[
d(a, a_1), \ d(a, a_2), \ \dots
\]
It is natural to define the sequence $(a_n)$ as converging to $a$ if the sequence of distances converges to zero.

\begin{definition}[see \cite{Mendelson, negrescu}]
Let $(X, d)$ be a metric space and $(a_n)_{n \ge 1}$ a sequence of points in $X$. A point $a \in X$ is called the \textbf{limit} of $(a_n)$ if
\[
\lim_{n \to \infty} d(a_n, a) = 0.
\]
We then write
\[
\lim_{n \to \infty} a_n = a
\]
and say that $(a_n)$ \textbf{converges} to $a$.
\end{definition}

\begin{remark}
Equivalently, $(a_n)$ converges to $a$ if for every neighborhood $V$ of $a$, there exists an integer $N$ such that
\[
n > N \implies a_n \in V.
\]
\end{remark}

\begin{example}[Convergence in $\mathbb{R}$]
Consider the sequence $a_n = 1/n$ in $\mathbb{R}$ with the usual distance $d(x,y) = |x-y|$. For any $\epsilon > 0$, choose $N$ such that $1/N < \epsilon$. Then for all $n > N$, we have $|a_n - 0| = 1/n < \epsilon$, so $\lim_{n \to \infty} a_n = 0$.
\end{example}

\begin{example}[Convergence in $\mathbb{R}^2$]
Let $a_n = (1/n, (-1)^n/n)$ in $\mathbb{R}^2$ with the Euclidean metric. The distance to the origin is
\[
d(a_n, (0,0)) = \sqrt{(1/n)^2 + ((-1)^n/n)^2} = \sqrt{2}/n \to 0.
\]
Thus $a_n \to (0,0)$ as $n \to \infty$.
\end{example}

\begin{proof}[Neighborhood formulation of convergence]
Suppose $\lim_{n \to \infty} a_n = a$. Let $V$ be a neighborhood of $a$. There exists $\epsilon > 0$ such that $B(a;\epsilon) \subset V$. By convergence, there is an $N$ such that $n > N$ implies $a_n \in B(a;\epsilon) \subset V$.  

Conversely, if for every neighborhood $V$ of $a$, almost all $a_n$ are in $V$, then taking $V = B(a;\epsilon)$ for any $\epsilon > 0$ gives an $N$ such that $n > N$ implies $d(a_n, a) < \epsilon$. Hence $\lim_{n \to \infty} a_n = a$.
\end{proof}

\begin{remark}
If $S$ is an infinite set and a statement fails only for finitely many elements, we say it holds for \emph{almost all} elements. Therefore, $(a_n) \to a$ if for each neighborhood $V$ of $a$, almost all $a_n$ lie in $V$.
\end{remark}

\begin{remark}
Continuity of a function can be equivalently described using sequences: a function $f$ is continuous at $a$ if whenever $a_n \to a$, it follows that $f(a_n) \to f(a)$.
\end{remark}

\begin{theorem}[see \cite{Mendelson}]
Let $(X, d), (Y, d')$ be metric spaces. A function $f: X \to Y$ is continuous at a point $a \in X$ if and only if for every sequence $a_1, a_2, \ldots$ in $X$ that converges to $a$, the sequence $f(a_1), f(a_2), \ldots$ converges to $f(a)$.
\end{theorem}

\begin{proof}
Suppose $f$ is continuous at $a$ and let $a_n \to a$. For any neighborhood $V$ of $f(a)$, $f^{-1}(V)$ is a neighborhood of $a$. By the definition of convergence, there exists $N$ such that $a_n \in f^{-1}(V)$ for $n > N$. Hence $f(a_n) \in V$ for $n > N$, so $f(a_n) \to f(a)$.

Conversely, if $f$ were not continuous at $a$, there would exist a neighborhood $V$ of $f(a)$ such that for every neighborhood $U$ of $a$, $f(U) \not\subset V$. Choosing $U = B(a; 1/n)$ for $n = 1, 2, \ldots$, we can pick $a_n \in B(a; 1/n)$ with $f(a_n) \notin V$. Then $a_n \to a$ but $f(a_n) \not\to f(a)$, a contradiction.
\end{proof}

\begin{remark}
If $\lim_{n \to \infty} a_n = a$, we may write 
\[
\lim_{n \to \infty} f(a_n) = f(\lim_{n \to \infty} a_n),
\] 
so a continuous function can be viewed as one that commutes with taking limits.
\end{remark}

\begin{example}[Sequence in $\mathbb{R}$, see \cite{Mendelson}]
Consider $a_n = 1/n$. Then $a_n \to 0$ because for any $\varepsilon > 0$, we can choose $N > 1/\varepsilon$ so that $|a_n - 0| = 1/n < \varepsilon$ for $n > N$.
\end{example}

\begin{example}[Continuous function, see \cite{Mendelson}]
Let $f(x) = x^2$ and $a = 2$. For a sequence $x_n \to 2$, we have $f(x_n) = x_n^2 \to 4 = f(2)$, illustrating the theorem.
\end{example}

\begin{example}[Discontinuous function, see \cite{Mendelson}]
Let 
\[
f(x) =
\begin{cases}
0, & x < 0 \\
1, & x \ge 0
\end{cases}
\]
and $a = 0$. Take $x_n = -1/n \to 0$. Then $f(x_n) = 0 \not\to f(0) = 1$, so $f$ is not continuous at $0$.
\end{example}

\begin{definition}[see \cite{Mendelson}]
Let $A \subset \mathbb{R}$. A number $b$ is an \textbf{upper bound} of $A$ if $x \le b$ for all $x \in A$, and $c$ is a \textbf{lower bound} if $c \le x$ for all $x \in A$. If both exist, $A$ is called \textbf{bounded}.

An upper bound $b^*$ is the \textbf{least upper bound} (l.u.b.) of $A$ if $b^* \le b$ for every upper bound $b$ of $A$. Similarly, a lower bound $c^*$ is the \textbf{greatest lower bound} (g.l.b.) if $c \le c^*$ for every lower bound $c$ of $A$.
\end{definition}

\begin{example}[Bounds in $\mathbb{R}$, see \cite{Mendelson}]
Let $A = \{1/n : n \in \mathbb{N}\}$. Then $1$ is an upper bound, $0$ is a lower bound, $\sup A = 1$, and $\inf A = 0$.
\end{example}

Not every set of real numbers is bounded above. One of the fundamental properties of $\mathbb{R}$, known as the completeness property, states that any non-empty set $A \subset \mathbb{R}$ that has an upper bound must have a least upper bound. Similarly, if a non-empty set $B \subset \mathbb{R}$ has a lower bound, consider the set of negatives of elements of $B$. This set has an upper bound, hence a l.u.b., whose negative gives the g.l.b. of $B$. Therefore, every non-empty set of real numbers with a lower bound has a greatest lower bound.

\begin{remark}
The greatest lower bound of a set $A$ may or may not belong to $A$. For example, $0$ is the g.l.b. of $[0,1]$ and $0 \in [0,1]$, whereas $0$ is also the g.l.b. of $(0,1)$ but $0 \notin (0,1)$. In any case, the g.l.b. can be approximated arbitrarily closely by elements of the set.
\end{remark}

\begin{lemma}[see \cite{Mendelson}]
Let $b$ be the greatest lower bound of a non-empty set $A \subset \mathbb{R}$. Then for every $\epsilon > 0$, there exists an element $x \in A$ such that
\[
x - b < \epsilon.
\]
\end{lemma}

\begin{proof}
Assume for contradiction that there exists $\epsilon > 0$ such that $x - b \ge \epsilon$ for all $x \in A$. Then $b + \epsilon$ would be a lower bound of $A$, contradicting the fact that $b$ is the g.l.b.
\end{proof}

\begin{corollary}[see \cite{Mendelson}]
If $b$ is the greatest lower bound of a non-empty set $A \subset \mathbb{R}$, then there exists a sequence $(a_n)$ with $a_n \in A$ for all $n$ such that
\[
\lim_{n \to \infty} a_n = b.
\]
\end{corollary}

\begin{proof}
Take $\epsilon = 1/n$. By the lemma, there exists $a_n \in A$ such that $0 \le a_n - b < 1/n$. Hence, $a_n \to b$ as $n \to \infty$.
\end{proof}

\begin{example}[Simple sequence approaching g.l.b., see \cite{Mendelson}]
Let $A = \{1/n : n \in \mathbb{N}\}$. Then $\inf A = 0$, and the sequence $a_n = 1/n \in A$ converges to $0$.
\end{example}

\begin{definition}[see \cite{Mendelson}]
Let $(X, d)$ be a metric space, $a \in X$, and $A \subset X$ a non-empty set. The \textbf{distance from $a$ to $A$} is defined as the greatest lower bound of the set $\{d(a, x) : x \in A\}$ and is denoted by $d(a, A)$.
\end{definition}

\begin{corollary}[see \cite{Mendelson}]
Let $(X, d)$ be a metric space, $a \in X$, and $A \subset X$ non-empty. Then there exists a sequence $(a_n)$ of points in $A$ such that
\[
\lim_{n \to \infty} d(a, a_n) = d(a, A).
\]
\end{corollary}

\begin{example}[Distance from a point to a set, see \cite{Mendelson}]
In $\mathbb{R}$ with the usual metric, let $a = 0$ and $A = [1,2]$. Then $d(0, A) = 1$, and the sequence $a_n = 1 + 1/n \in A$ satisfies $d(0, a_n) \to 1$.
\end{example}

\newpage
\section{Open Sets and Closed Sets}

As we saw in a previous chapter, the open ball \( B(a; \delta) \) is a neighborhood of each of its points in a metric space. The collection of subsets that have this property has an important role in the study of topology.

\begin{definition}[see \cite{Mendelson}]
Let $(X, d)$ be a metric space. A subset $O \subset X$ is called \textbf{open} if every point $x \in O$ has a neighborhood contained entirely within $O$; that is, for each $x \in O$ there exists $\varepsilon > 0$ such that the open ball $B(x; \varepsilon) \subset O$.
\end{definition}

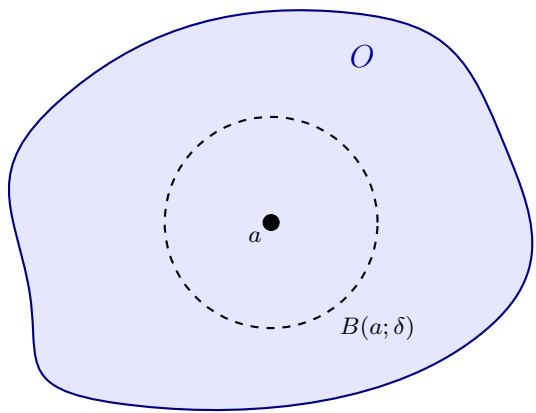
\begin{figure}[h!]
\centering
\begin{tikzpicture}[scale=2.0]

% Open set O - irregular filled region
\fill[blue!10!white] 
  plot[smooth cycle, tension=0.9] coordinates {
    (0,0) (2.3,0.4) (2.5,1.8) (1.4,2.6) (-0.4,2.0) (-0.6,0.8)
  };

% Boundary of the open set (for clarity)
\draw[blue!50!black, thick] 
  plot[smooth cycle, tension=0.9] coordinates {
    (0,0) (2.3,0.4) (2.5,1.8) (1.4,2.6) (-0.4,2.0) (-0.6,0.8)
  };

% Label for open set O
\node[blue!60!black] at (1.6,2.3) {\large $O$};

% Point a inside O
\filldraw[black] (1,1.2) circle (1.5pt);
\node[below left] at (1,1.2) {\small $a$};

% Open ball centered at a
\draw[dashed, thick] (1,1.2) circle (0.7);
\node at (1.7,0.5) {\small $B(a;\delta)$};

\end{tikzpicture}
\caption{An open set \( O \) in a metric space that contains an open ball \( B(a; \delta) \) around each of its points.}
\end{figure}

\begin{remark}
Intuitively, a set is open if none of its points are ``on the edge'' of the set. Around every point, no matter which one we pick, we can always find a small open ball that still lies completely inside the set. For example, in $\mathbb{R}$ with the usual metric, the interval $(0,1)$ is open because for each $x \in (0,1)$, we can take a small interval $(x - \varepsilon, x + \varepsilon)$ that stays inside $(0,1)$. In contrast, $[0,1]$ is not open, since near $0$ or $1$ there is no such interval fully contained in it.
\end{remark}

\begin{theorem}[see \cite{Mendelson}]
A subset $O$ of a metric space $(X, d)$ is open if and only if it can be expressed as a union of open balls.
\end{theorem}

\begin{proof}
Assume first that $O$ is open. Then, by definition, for every $a \in O$ there exists $\delta_a > 0$ such that $B(a; \delta_a) \subset O$. Consequently,
\[
O = \bigcup_{a \in O} B(a; \delta_a),
\]
so $O$ is indeed a union of open balls.

Conversely, suppose that $O$ can be written as a union of open balls:
\[
O = \bigcup_{\alpha \in I} B(a_\alpha; \delta_\alpha),
\]
where $I$ is an indexing set. If $x \in O$, then $x$ belongs to one of these balls, say $B(a_\beta; \delta_\beta)$. Since $B(a_\beta; \delta_\beta)$ is itself open, it is a neighborhood of $x$, and moreover $B(a_\beta; \delta_\beta) \subset O$. Hence $O$ is a neighborhood of each of its points and therefore open.
\end{proof}

\begin{example}
Consider the real line $\mathbb{R}$ with the usual distance $d(x, y) = |x - y|$. The open interval $(0,1)$ can be written as the union of open balls $B(x; r_x)$ where $r_x = \min\{x, 1 - x\}$. Indeed, for every $x \in (0,1)$, this radius ensures that the ball $B(x; r_x)$ is entirely contained within $(0,1)$. This example illustrates concretely how an open set can be built as a union of open balls centered at each of its points.
\end{example}

\begin{remark}
The characterization of open sets as unions of open balls is not only elegant, but also practical. It means that in a metric space, the open balls form a \emph{basis} for the topology: every open set can be constructed by combining these building blocks. This viewpoint is central to the way topology generalizes concepts of openness beyond metric spaces.
\end{remark}

\begin{theorem}[see \cite{Kaplansky}]
Every open ball in a metric space is an open set.
\end{theorem}

\begin{proof}
Let $(X,d)$ be a metric space and fix $a\in X$ and $r>0$. Take any $x\in B(a;r)$, so $d(a,x)<r$, and put $\varepsilon:=r-d(a,x)>0$. If $y$ satisfies $d(x,y)<\varepsilon$ then by the triangle inequality
\[
d(a,y)\le d(a,x)+d(x,y)<d(a,x)+\varepsilon=r,
\]
hence $y\in B(a;r)$. Thus $B(x;\varepsilon)\subset B(a;r)$, so $B(a;r)$ is a neighborhood of each of its points and therefore open.
\end{proof}

\begin{theorem}[see \cite{Kaplansky}]
Any union of open sets in a metric space is open.
\end{theorem}

\begin{proof}
Let $\{O_\alpha\}_{\alpha\in A}$ be a family of open subsets of $(X,d)$ and set $O=\bigcup_{\alpha\in A}O_\alpha$. 
If $A=\varnothing$ then $O=\varnothing$, which is open; otherwise, pick $x\in O$. Then $x\in O_{\alpha_0}$ for some $\alpha_0\in A$, and since $O_{\alpha_0}$ is open there exists $\varepsilon>0$ with $B(x;\varepsilon)\subset O_{\alpha_0}\subset O$. Hence every $x\in O$ has a neighborhood contained in $O$, so $O$ is open.
\end{proof}

\begin{corollary}
A finite intersection of open sets is open.
\end{corollary}

\begin{proof}
For two sets the proof is direct: if $U,V$ are open and $x\in U\cap V$, choose $\varepsilon_1,\varepsilon_2>0$ with $B(x;\varepsilon_1)\subset U$ and $B(x;\varepsilon_2)\subset V$. Then $B(x;\min\{\varepsilon_1,\varepsilon_2\})\subset U\cap V$. The general finite case follows by induction.
\end{proof}

\begin{remark}
As immediate consequences we have that $\varnothing$ and $X$ are open, and that arbitrary unions and finite intersections of open sets remain open. Infinite intersections of open sets need not be open (e.g.\ $\bigcap_{n\ge1}(-1/n,1/n)=\{0\}$ in $\mathbb{R}$).
\end{remark}

\begin{theorem}[see \cite{Mendelson}]
Let $f:(X,d)\to (Y,d')$ be a map between metric spaces. Then $f$ is continuous on $X$ if and only if for every open set $O\subset Y$ the preimage $f^{-1}(O)$ is open in $X$.
\end{theorem}

\begin{proof}
$(\Rightarrow)$ Suppose $f$ is continuous. Let $O\subset Y$ be open and take $x\in f^{-1}(O)$, so $f(x)\in O$. Since $O$ is a neighborhood of $f(x)$ there exists $\varepsilon>0$ with $B_Y(f(x);\varepsilon)\subset O$. By continuity at $x$ there is $\delta>0$ such that $d(x,y)<\delta$ implies $d'(f(x),f(y))<\varepsilon$, so $f(B_X(x;\delta))\subset B_Y(f(x);\varepsilon)\subset O$. Hence $B_X(x;\delta)\subset f^{-1}(O)$ and $f^{-1}(O)$ is open.

$(\Leftarrow)$ Conversely, assume preimages of open sets are open. Fix $a\in X$ and $\varepsilon>0$. Then $B_Y(f(a);\varepsilon)$ is open, so $U:=f^{-1}(B_Y(f(a);\varepsilon))$ is open and contains $a$. Thus there exists $\delta>0$ with $B_X(a;\delta)\subset U$. For $x$ with $d(x,a)<\delta$ we have $f(x)\in B_Y(f(a);\varepsilon)$, i.e.\ $d'(f(x),f(a))<\varepsilon$. Hence $f$ is continuous at $a$. Since $a$ was arbitrary, $f$ is continuous on $X$.
\end{proof}

\begin{example}
As an illustration, consider the projection \(\pi:\mathbb{R}^2\to\mathbb{R}\), \(\pi(x,y)=x\). If \(O\subset\mathbb{R}\) is open then \(\pi^{-1}(O)=O\times\mathbb{R}\) is open in \(\mathbb{R}^2\); hence \(\pi\) is continuous by the theorem above. This matches the usual \(\varepsilon\)-\(\delta\) verification but the preimage criterion is often shorter and more conceptual.
\end{example}

Just as the collection of neighborhoods of points in a metric space satisfies certain fundamental properties, so does the collection of open sets. 
These properties are essential because they describe the algebraic structure of openness — the way open sets behave under union, intersection, and complementation.

\begin{theorem}[see \cite{Mendelson}]
Let \( (X, d) \) be a metric space. Then the collection of open subsets of \( X \) satisfies the following properties:

\begin{enumerate}
    \item[\( O_1 \).] The empty set \( \varnothing \) is open.
    \item[\( O_2 \).] The entire space \( X \) is open.
    \item[\( O_3 \).] The intersection of finitely many open sets is open; that is, if \( O_1, O_2, \ldots, O_n \) are open, then
    \[
    O_1 \cap O_2 \cap \cdots \cap O_n
    \]
    is open.
    \item[\( O_4 \).] The union of any (finite or infinite) family of open sets is open; that is, if \( \{O_\alpha\}_{\alpha \in I} \) is a family of open sets, then
    \[
    \bigcup_{\alpha \in I} O_\alpha
    \]
    is open.
\end{enumerate}
\end{theorem}

\begin{proof}
\textbf{(\( O_1 \))} The empty set is open because there is no point in \( \varnothing \) at which the definition of openness could fail.  

\textbf{(\( O_2 \))} Given any \( a \in X \) and any \( \varepsilon > 0 \), the open ball \( B(a;\varepsilon) \subseteq X \), so \( X \) is a neighborhood of each of its points. Thus \( X \) is open.

\textbf{(\( O_3 \))} Let \( a \in O_1 \cap O_2 \cap \cdots \cap O_n \), where each \( O_i \) is open. Then each \( O_i \) contains a ball \( B(a;\varepsilon_i) \subseteq O_i \). Taking 
\[
\varepsilon = \min\{\varepsilon_1, \varepsilon_2, \ldots, \varepsilon_n\},
\]
we obtain \( B(a;\varepsilon) \subseteq O_1 \cap O_2 \cap \cdots \cap O_n \). Hence the intersection is a neighborhood of each of its points and is therefore open.

\textbf{(\( O_4 \))} Let \( O = \bigcup_{\alpha \in I} O_\alpha \), with each \( O_\alpha \) open. If \( x \in O \), then \( x \in O_{\alpha_0} \) for some \( \alpha_0 \in I \), and there exists \( \varepsilon > 0 \) such that \( B(x;\varepsilon) \subseteq O_{\alpha_0} \subseteq O \). Hence \( O \) is a neighborhood of each of its points and is open.
\end{proof}

\begin{definition}[see \cite{Mendelson}]
A subset \( F \) of a metric space \( (X, d) \) is said to be \textbf{closed} if its complement, \( C(F) = X \setminus F \), is open.
\end{definition}

The notions of open and closed are not opposites in the usual sense: a set may be both, or neither.  
For example, in the real number system, the closed interval \([a,b]\) is closed because its complement is the union of the two open sets
\[
(-\infty, a) \quad \text{and} \quad (b, \infty).
\]
On the other hand, the open interval \((a,b)\) is open because every point inside it lies at the center of some smaller open subinterval.

A common misconception is that a set cannot be both open and closed.  
However, in any metric space \((X,d)\), both \( X \) and \( \varnothing \) are open, and since the complement of an open set is closed, both \( X \) and \( \varnothing \) are also closed.  
Thus, these two sets are simultaneously open and closed (often called *clopen* sets).

Whether other nontrivial subsets of \( X \) share this property depends on the topology of the space; for instance, a metric space is called \textbf{connected} if the only clopen subsets are \( X \) and \( \varnothing \).

Finally, it is worth noting that some sets are neither open nor closed.  
For example, in \( \mathbb{R} \),
\[
A = [0,1)
\]
is not open because 0 does not have a neighborhood contained in \( A \), and it is not closed because its complement does not contain all its limit points (1 is a limit point of \( A \), yet \( 1 \notin A \)).

The concepts of openness and closedness are deeply connected through the idea of \emph{limit points}.  
A limit point represents, in an intuitive sense, a point that can be "approached" by elements of a given set without necessarily belonging to that set.

\begin{definition}[see \cite{Mendelson}]
Let \( A \subset X \) be a subset of a metric space \( (X, d) \).  
A point \( b \in X \) is called a \textbf{limit point} (or \textbf{accumulation point}) of \( A \) if every neighborhood of \( b \) contains at least one point of \( A \) distinct from \( b \) itself.
\end{definition}

In other words, no matter how small a ball we draw around \( b \), we always find some point of \( A \) (other than \( b \)) inside it.  
This captures the idea that points of \( A \) can get arbitrarily close to \( b \).

If \( b \) is a limit point of \( A \), then for each \( n \in \mathbb{N} \), the open ball \( B\!\left(b; \tfrac{1}{n}\right) \) contains some point \( a_n \in A \setminus \{b\} \).  
Thus the sequence \( (a_n) \) satisfies \( \lim_{n \to \infty} a_n = b \).  
Therefore, every limit point of \( A \) can be obtained as the limit of a sequence of distinct points of \( A \).

However, the converse is not always true.  
If \( b \in A \) and there exists some \( \delta > 0 \) such that \( B(b; \delta) \cap A = \{b\} \), then no other point of \( A \) approaches \( b \).  
In this case, although the constant sequence \( (b,b,b,\ldots) \) converges to \( b \), the point \( b \) is not a limit point of \( A \).  
Such a point is called an \textbf{isolated point} of \( A \).

\begin{example}
In \( \mathbb{R} \), consider \( A = \{0\} \cup (1,2) \).  
Every point of the open interval \( (1,2) \) is a limit point of \( A \), while \( 0 \) is an isolated point because there exists a ball \( B(0; \tfrac{1}{2}) \) that contains no point of \( A \) other than \( 0 \) itself.
\end{example}

\begin{theorem}[see \cite{Mendelson}]
In a metric space \( X \), a set \( F \subset X \) is closed if and only if it contains all its limit points.
\end{theorem}

\begin{proof}
Let \( F' \) denote the set of all limit points of \( F \).

\textbf{($\Rightarrow$)} Suppose \( F \) is closed. Then its complement \( C(F) = X \setminus F \) is open.  
Let \( b \notin F \). Since \( C(F) \) is open, there exists some \( \delta > 0 \) such that \( B(b; \delta) \subset C(F) \), which means that \( B(b; \delta) \cap F = \emptyset \).  
Hence no neighborhood of \( b \) meets \( F \) at a point other than \( b \), and thus \( b \) is not a limit point of \( F \).  
Therefore, \( F' \subset F \).

\textbf{($\Leftarrow$)} Conversely, suppose \( F' \subset F \).  
Let \( b \in C(F) = X \setminus F \). Then \( b \notin F' \), meaning there exists \( \delta > 0 \) such that \( B(b; \delta) \cap F = \emptyset \).  
Thus \( B(b; \delta) \subset C(F) \), showing that \( C(F) \) is open.  
Consequently, \( F \) is closed.
\end{proof}

\begin{remark}
This theorem is fundamental: it allows us to define closedness without referring to complements or open sets, but purely in terms of the internal structure of the set and its limit points.
\end{remark}

We can restate this result in terms of sequences, which is particularly intuitive in metric spaces.

\begin{theorem}[Sequential characterization of closed sets, see \cite{Mendelson}]
In a metric space \( (X, d) \), a set \( F \subset X \) is closed if and only if for every sequence \( (a_n) \) of points of \( F \) that converges to some point \( a \in X \), we have \( a \in F \).
\end{theorem}

\begin{proof}
\textbf{($\Rightarrow$)} Assume \( F \) is closed, and let \( (a_n) \subset F \) with \( \lim_{n \to \infty} a_n = a \).  
If infinitely many distinct \( a_n \)’s appear in the sequence, then every neighborhood of \( a \) contains infinitely many points of \( F \); hence \( a \) is a limit point of \( F \), and by the previous theorem \( a \in F \).  
If the sequence becomes constant after some index \( N \) (that is, \( a_n = a \) for all \( n > N \)), then trivially \( a \in F \).  
Therefore, in both cases, \( a \in F \).

\textbf{($\Leftarrow$)} Conversely, suppose that whenever a sequence \( (a_n) \subset F \) converges to \( a \in X \), we have \( a \in F \).  
Let \( b \) be a limit point of \( F \). Then by definition, there exists a sequence \( (a_n) \subset F \setminus \{b\} \) with \( a_n \to b \).  
By the given assumption, \( b \in F \).  
Thus \( F \) contains all its limit points, and hence \( F \) is closed.
\end{proof}

This sequential criterion is extremely useful in analysis, since many problems are more naturally expressed in terms of convergence rather than in terms of open or closed sets.

Finally, we can also describe closed sets in terms of the \emph{distance} from a point to a set — a viewpoint that connects topology with the metric structure itself.

\begin{theorem}[see \cite{Mendelson}]
A subset $F$ of a metric space $(X, d)$ is closed if and only if for every point $x \in X$, the condition $d(x, F) = 0$ implies $x \in F$.
\end{theorem}

\begin{proof}
Assume first that $F$ is closed. Let $x \in X$ satisfy $d(x, F) = 0$. By Corollary 5.9, there exists a sequence $(a_n)$ of points in $F$ such that $\lim_{n \to \infty} d(x, a_n) = 0$. This means that every neighborhood of $x$ intersects $F$. If $x = a_n$ for some $n$, then $x \in F$. Otherwise, $x$ is a limit point of $F$ and by Theorem 6.7, $x \in F$.

Conversely, suppose that $d(x, F) = 0$ implies $x \in F$. If $x$ is a limit point of $F$, then clearly $d(x, F) = 0$, so $x \in F$. Therefore, $F$ contains all its limit points and is closed.
\end{proof}

\begin{theorem}[see \cite{Mendelson}]
Let $(X, d)$ and $(Y, d^*)$ be metric spaces. A function $f: X \to Y$ is continuous if and only if for every closed subset $A$ of $Y$, the preimage $f^{-1}(A)$ is closed in $X$.
\end{theorem}

\begin{proof}
For any $A \subset Y$, note that $C(f^{-1}(A)) = f^{-1}(C(A))$. Continuity of $f$ is equivalent to the property that the preimage of every open set is open. Applying this to complements, it follows that the preimage of each closed set is closed.
\end{proof}

\begin{theorem}[see \cite{Mendelson}]
Let $(X, d)$ be a metric space.

\begin{enumerate}
    \item[\( C_1 \):] $X$ is closed.
    \item[\( C_2 \):] $\emptyset$ is closed.
    \item[\( C_3 \):] The union of finitely many closed sets is closed.
    \item[\( C_4 \):] The intersection of any collection of closed sets is closed.
\end{enumerate}
\end{theorem}

\begin{proof}
The facts $C_1$ and $C_2$ are immediate. The statements $C_3$ and $C_4$ follow directly from DeMorgan's laws applied to the corresponding properties of open sets, $O_3$ and $O_4$.
\end{proof}

\begin{example}
The union of infinitely many closed sets need not be closed. For instance, consider $F_n = [1/n, 1]$ for each positive integer $n$. Then 
\[
\bigcup_{n=1}^{\infty} F_n = (0, 1],
\]
which is not closed because $0$ is a limit point of the union but is not contained in it.
\end{example}

\begin{theorem}[Normality]
Every pseudo-metric space is normal: any two disjoint closed subsets can be separated by disjoint open sets.
\end{theorem}

In this chapter we have explored the fundamental notions of open and closed sets in metric spaces. Open sets were introduced as neighborhoods of each of their points and were shown to be unions of open balls. Closed sets were characterized in multiple equivalent ways: via complements, limit points, sequences, and distance from points. We also saw how these concepts connect to continuity, with continuous functions preserving the openness or closedness of sets under preimages. 

Together, these ideas provide a foundation for topology, showing how local properties around points (neighborhoods, balls) lead to global properties of sets and functions, setting the stage for further study of connectedness, compactness, and more advanced topological concepts.

\newpage
\section{Subspaces and Equivalent Metrics}

Let $(X, d)$ be a metric space. Suppose $Y$ is a non-empty subset of $X$.  
We can naturally endow $Y$ with a metric by restricting the distance function $d$ to pairs of points in $Y$. In this way, each non-empty subset $Y \subseteq X$ inherits a metric space structure, denoted by $(Y, d|_{Y \times Y})$.  
Thus, every subset of a metric space may itself be viewed as a smaller metric space, where distances are measured exactly as in the original space.

Conversely, suppose we are given two metric spaces $(X, d)$ and $(Y, d')$ such that $Y \subseteq X$.  
It is natural to ask whether the metric $d'$ on $Y$ is simply the restriction of the ambient metric $d$ --- that is, whether the distances in $Y$ coincide with those inherited from $X$.

\begin{definition}[see \cite{Mendelson}]
Let $(X, d)$ and $(Y, d')$ be metric spaces.  
We say that $(Y, d')$ is a \textbf{subspace} of $(X, d)$ if:
\begin{enumerate}
    \item $Y \subseteq X$;
    \item $d' = d|_{Y \times Y}$.
\end{enumerate}
\end{definition}

Let $Y \subseteq X$ and denote by $i: Y \to X$ the inclusion mapping.  
Then $i \times i: Y \times Y \to X \times X$ is the corresponding inclusion on pairs, defined by $(i \times i)(y_1, y_2) = (y_1, y_2)$.  
In this context, we say that $(Y, d')$ is a subspace of $(X, d)$ precisely when the following relationship holds between the two metrics --- that is, when $d'$ is exactly the pullback of $d$ under the inclusion map.

\[
\begin{tikzcd}
Y \times Y \arrow[rr, "d'"] \arrow[d, "i \times i"'] & & \mathbb{R} \\
X \times X \arrow[rru, "d"']
\end{tikzcd}
\]

Every non-empty subset of a metric space naturally gives rise to its own subspace.  
Hence, the number of subspaces of a metric space $(X, d)$ is exactly the number of its non-empty subsets.  
Each of these subsets inherits the same notion of distance as in the original space.

\begin{example}[see \cite{Mendelson}]
Let $\mathbb{Q}$ denote the set of rational numbers, and define the metric
\[
d_{\mathbb{Q}}(a, b) = |a - b|, \quad \text{for all } a, b \in \mathbb{Q}.
\]
Then $(\mathbb{Q}, d_{\mathbb{Q}})$ is a subspace of the real line $(\mathbb{R}, d)$, where $d$ is the usual Euclidean metric.  
This example illustrates how the rationals, though dense in $\mathbb{R}$, form their own metric space with distances measured exactly as in $\mathbb{R}$.
\end{example}

\begin{example}[see \cite{Mendelson}]
Let $I^n$ denote the unit $n$-cube, that is,
\[
I^n = \{ (x_1, x_2, \ldots, x_n) \in \mathbb{R}^n : 0 \le x_i \le 1 \text{ for all } i = 1, 2, \ldots, n \}.
\]
Define a metric $d_I : I^n \times I^n \to \mathbb{R}$ by
\[
d_I(x, y) = \max\{ |x_i - y_i| : i = 1, \ldots, n \}.
\]
Then $(I^n, d_I)$ is a subspace of $(\mathbb{R}^n, d)$, where $d$ denotes the same maximum (or uniform) metric on $\mathbb{R}^n$.  
Geometrically, this means that distances between points in the cube are measured exactly as in the ambient space, but only within the cube’s boundaries.
\end{example}

\begin{example}[see \cite{Mendelson}]
Let $S^n$ denote the $n$-sphere, defined as the set
\[
S^n = \{ (x_1, x_2, \ldots, x_{n+1}) \in \mathbb{R}^{n+1} : x_1^2 + x_2^2 + \cdots + x_{n+1}^2 = 1 \}.
\]
Define $d_S : S^n \times S^n \to \mathbb{R}$ by
\[
d_S((x_1, \ldots, x_{n+1}), (y_1, \ldots, y_{n+1})) 
= \sqrt{\sum_{i=1}^{n+1} (x_i - y_i)^2}.
\]
Then $(S^n, d_S)$ is a subspace of the Euclidean space $(\mathbb{R}^{n+1}, d)$.  
Here the metric $d_S$ coincides with the Euclidean metric restricted to the sphere, showing that the sphere inherits its geometry from the ambient space.
\end{example}

\begin{example}[see \cite{Mendelson}]
Let $A$ be the set of all $(n + 1)$-tuples of real numbers satisfying $x_{n+1} = 0$, that is,
\[
A = \{ (x_1, x_2, \ldots, x_{n+1}) \in \mathbb{R}^{n+1} : x_{n+1} = 0 \}.
\]
Define $d_A : A \times A \to \mathbb{R}$ by
\[
d_A((x_1, \ldots, x_n, 0), (y_1, \ldots, y_n, 0)) = \max\{ |x_i - y_i| : i = 1, \ldots, n \}.
\]
Then $(A, d_A)$ is a subspace of $(\mathbb{R}^{n+1}, d)$, where $d$ again denotes the maximum metric.  
Geometrically, $A$ can be viewed as an $n$-dimensional hyperplane sitting inside $\mathbb{R}^{n+1}$, inheriting its distances directly from the surrounding space.
\end{example}

\begin{theorem}[see \cite{Mendelson}]
Let $(Y, d')$ be a subspace of a metric space $(X, d)$. Then the inclusion mapping $i : Y \rightarrow X$ is continuous.
\end{theorem}

\begin{proof}
Let $a \in Y$ and $\varepsilon > 0$.  
Choose $\delta = \varepsilon$.  
If $d'(a, y) < \delta$ for some $y \in Y$, then by the definition of a subspace metric,
\[
d(i(a), i(y)) = d(a, y) = d'(a, y) < \delta = \varepsilon.
\]
Hence, $i$ is continuous at every point $a \in Y$, and therefore $i$ is continuous on $Y$.
\end{proof}

The metric space $(A, d_A)$ introduced in Example 4 is, in most respects, a copy of the metric space $(\mathbb{R}^n, d)$.  
The only distinction is notational: a point of $\mathbb{R}^n$ is an $n$-tuple of real numbers, while a point of $A$ is an $(n + 1)$-tuple whose last coordinate is zero.  
The correspondence between these two spaces exemplifies the concept of \emph{metric equivalence} (or \emph{isometry}) between metric spaces.

\begin{definition}[see \cite{Mendelson}]
Two metric spaces $(A, d_A)$ and $(B, d_B)$ are said to be \textbf{metrically equivalent} or \textbf{isometric} if there exist mutually inverse functions 
\[
f: A \to B \quad \text{and} \quad g: B \to A
\]
such that:
\[
d_B(f(x), f(y)) = d_A(x, y) \quad \text{for all } x, y \in A,
\]
and
\[
d_A(g(u), g(v)) = d_B(u, v) \quad \text{for all } u, v \in B.
\]
In this case, we say that the \textbf{isometry} (or metric equivalence) is defined by the pair of mappings $f$ and $g$.
\end{definition}

\begin{theorem}[see \cite{Mendelson}]
A necessary and sufficient condition for two metric spaces $(A, d_A)$ and $(B, d_B)$ to be metrically equivalent is the existence of a function $f: A \to B$ such that:
\begin{enumerate}
    \item $f$ is one-to-one;
    \item $f$ is onto;
    \item for all $x, y \in A$, we have $d_B(f(x), f(y)) = d_A(x, y)$.
\end{enumerate}
\end{theorem}

\begin{proof}
The necessity of these conditions is clear, since if $f: A \to B$ and $g: B \to A$ define a metric equivalence, then $f$ must be both one-to-one and onto, and the distance-preserving property holds by definition.

Conversely, suppose there exists a mapping $f: A \to B$ satisfying the three stated properties.  
Then $f$ is invertible, and its inverse $g: B \to A$ is defined by $g(b) = a$ whenever $f(a) = b$.  
For any $u, v \in B$, set $x = g(u)$ and $y = g(v)$. Then:
\[
d_A(g(u), g(v)) = d_A(x, y) = d_B(f(x), f(y)) = d_B(u, v).
\]
Hence, $f$ and $g$ define a metric equivalence between $(A, d_A)$ and $(B, d_B)$.
\end{proof}

Given metric spaces $(A, d_A)$ and $(B, d_B)$ together with functions $f: A \to B$ and $g: B \to A$,  
we define the product mappings
\[
f \times f : A \times A \to B \times B, \quad (f \times f)(x, y) = (f(x), f(y))
\]
for all $x, y \in A$, and similarly,
\[
g \times g : B \times B \to A \times A, \quad (g \times g)(u, v) = (g(u), g(v))
\]
for all $u, v \in B$.

The condition that
\[
d_B(f(x), f(y)) = d_A(x, y), \quad \text{for all } x, y \in A,
\]
is equivalent to saying that the following diagram is commutative:

\[
\begin{tikzcd}
A \times A \arrow[rr, "d_A"] \arrow[d, "f \times f"'] & & \mathbb{R} \\
B \times B \arrow[rru, "d_B"']
\end{tikzcd}
\]

In other words, the function $f : A \to B$ is \emph{distance-preserving}.  
From a diagrammatic point of view, the statement that $(A, d_A)$ and $(B, d_B)$ are metrically equivalent means that there exist functions 
\[
f : A \to B, \qquad g : B \to A,
\]
such that the following four diagrams commute:

\[
\begin{array}{cc}
% --- Diagrama 2 ---
\begin{tikzcd}
A \arrow[rr, "i_A"] \arrow[d, "f"'] & & A \\
B \arrow[rru, "g"']
\end{tikzcd}
&
% --- Diagrama 3 ---
\begin{tikzcd}
B \arrow[rr, "i_B"] \arrow[d, "g"'] & & B \\
A \arrow[rru, "f"']
\end{tikzcd}
\\[2em]
% --- Diagrama 4 ---
\begin{tikzcd}
A \times A \arrow[rr, "d_A"] \arrow[d, "f \times f"'] & & \mathbb{R} \\
B \times B \arrow[rru, "d_B"']
\end{tikzcd}
&
% --- Diagrama 5 ---
\begin{tikzcd}
B \times B \arrow[rr, "d_B"] \arrow[d, "g \times g"'] & & \mathbb{R} \\
A \times A \arrow[rru, "d_A"']
\end{tikzcd}
\end{array}
\]

Here, $i_A : A \to A$ and $i_B : B \to B$ denote the identity mappings.  
The first two diagrams express that $f$ and $g$ are inverse functions, while the last two express that both $f$ and $g$ preserve distances.  
Since the distance between $x$ and $y$ in $A$ equals the distance between $f(x)$ and $f(y)$ in $B$, it follows that $f$ is continuous.  
Similarly, $g$ is continuous.  

\begin{lemma}[see \cite{Mendelson}]
Let a metric equivalence between $(A, d_A)$ and $(B, d_B)$ be defined by inverse functions $f: A \to B$ and $g: B \to A$.  
Then both $f$ and $g$ are continuous.
\end{lemma}

From the perspective of studying continuity, the notion of \emph{metric equivalence} is somewhat restrictive.  
Indeed, requiring that two functions preserve distances exactly is often stronger than necessary when the goal is to understand continuity and open sets.  
We are therefore led to introduce a broader concept—one that retains the idea of correspondence between points but drops the condition of exact distance preservation.  
In other words, we no longer require the commutativity of the last pair of diagrams introduced earlier, but only that the first two diagrams commute and that the corresponding functions are continuous.

\begin{definition}[see \cite{Mendelson}]
Two metric spaces $(A, d_A)$ and $(B, d_B)$ are said to be \textbf{topologically equivalent} if there exist inverse functions 
$f : A \to B$ and $g : B \to A$ such that both $f$ and $g$ are continuous.  
In this case, the pair $(f, g)$ is said to define the \textbf{topological equivalence} between $(A, d_A)$ and $(B, d_B)$.
\end{definition}

As a direct consequence of Lemma~7.5, we obtain the following result.

\begin{corollary}[see \cite{Mendelson}]
Two metric spaces that are metrically equivalent are also topologically equivalent.
\end{corollary}

The converse of this statement is, however, false.  
That is, there exist metric spaces that are topologically equivalent but not metrically equivalent.  
A simple example can be found in Euclidean geometry:  
a circle of radius \(1\) and a circle of radius \(2\) (considered as subspaces of $(\mathbb{R}^2, d)$) are topologically equivalent—since there exists a continuous bijection with a continuous inverse between them—but they are not metrically equivalent, as the distances between corresponding points are not preserved.

To establish a sufficient condition for topological equivalence between metric spaces that share the same underlying set, we consider the following lemma.

\begin{lemma}[see \cite{Mendelson}]
Let $(X, d_1)$ and $(X, d_2)$ be two metric spaces.  
If there exists a constant \( K > 0 \) such that for all \( x, y \in X \),
\[
d_2(x, y) \leq K \cdot d_1(x, y),
\]
then the identity mapping \( i : (X, d_1) \to (X, d_2) \) is continuous.
\end{lemma}

\begin{proof}
Fix \( a \in X \) and let \( \varepsilon > 0 \).  
Set \( \delta = \varepsilon / K \).  
If \( d_1(x, a) < \delta \), then
\[
d_2(x, a) \leq K \cdot d_1(x, a) < K \delta = \varepsilon.
\]
Hence, the identity map is continuous at every point \( a \in X \).
\end{proof}

\begin{corollary}[see \cite{Mendelson}]
Let $(X, d)$ and $(X, d')$ be two metric spaces defined on the same set \( X \).  
If there exist positive constants \( K \) and \( K' \) such that for all \( x, y \in X \),
\[
d'(x, y) \leq K \cdot d(x, y), \qquad 
d(x, y) \leq K' \cdot d'(x, y),
\]
then the identity mappings define a topological equivalence between $(X, d)$ and $(X, d')$.
\end{corollary}

As a concrete example, consider the two metric spaces $(\mathbb{R}^n, d)$ and $(\mathbb{R}^n, d')$,  
where \( d \) denotes the \emph{maximum metric}
\[
d(x, y) = \max_{1 \leq i \leq n} |x_i - y_i|,
\]
and \( d' \) denotes the \emph{Euclidean metric}
\[
d'(x, y) = \sqrt{\sum_{i=1}^n (x_i - y_i)^2}.
\]
For every pair of points \( x, y \in \mathbb{R}^n \), the inequalities
\[
d(x, y) \leq d'(x, y) \leq \sqrt{n}\, d(x, y)
\]
hold.  
Thus, by Corollary~7.9, the metric spaces $(\mathbb{R}^n, d)$ and $(\mathbb{R}^n, d')$ are \emph{topologically equivalent}.

\begin{theorem}[see \cite{Mendelson}]
Let $(X, d)$ and $(Y, d')$ be two metric spaces, and let $f : X \to Y$ and $g : Y \to X$ be inverse functions.  
Then the following four statements are equivalent:
\begin{enumerate}
    \item The functions $f$ and $g$ are continuous;
    \item A subset $O \subseteq X$ is open if and only if $f(O)$ is an open subset of $Y$;
    \item A subset $F \subseteq X$ is closed if and only if $f(F)$ is a closed subset of $Y$;
    \item For each $a \in X$ and subset $N \subseteq X$, $N$ is a neighborhood of $a$ if and only if $f(N)$ is a neighborhood of $f(a)$.
\end{enumerate}
\end{theorem}

\begin{proof}
\textbf{$(1) \Rightarrow (2)$.}  
Suppose $O$ is an open subset of $X$.  
Then $f(O) = g^{-1}(O)$ is open in $Y$, since $g$ is continuous.  
Conversely, if $f(O)$ is open in $Y$, then $O = f^{-1}(f(O))$ is open in $X$, since $f$ is continuous.

\medskip

\textbf{$(2) \Rightarrow (4)$.}  
Let $a \in X$ and $N \subseteq X$.  
Then $N$ is a neighborhood of $a$ if and only if $N$ contains an open set $O$ such that $a \in O$.  
By $(2)$, $f(O)$ is an open subset of $Y$ containing $f(a)$, and $f(O) \subseteq f(N)$.  
Hence, $f(N)$ is a neighborhood of $f(a)$ in $Y$.

\medskip

\textbf{$(4) \Rightarrow (1)$.}  
Let $a \in X$ and let $U$ be a neighborhood of $f(a)$ in $Y$.  
Then, by $(4)$, $f^{-1}(U)$ is a neighborhood of $a$ in $X$.  
Thus, $f$ is continuous.  
Similarly, for any $b \in Y$ and neighborhood $V$ of $g(b)$, we have $g^{-1}(V) = f(V)$, which is a neighborhood of $f(g(b)) = b$; hence, $g$ is continuous.

\medskip

Therefore, statements $(1)$, $(2)$, and $(4)$ are equivalent.  
Finally, the equivalence between $(2)$ and $(3)$ follows immediately from the fact that a set is closed if and only if its complement is open.
\end{proof}

Statement $(1)$ in Theorem~7.10 asserts precisely that the metric spaces $(X, d)$ and $(Y, d')$ are \emph{topologically equivalent}.  
Consequently, Theorem~7.10 shows that two metric spaces are topologically equivalent if and only if there exist inverse functions that establish any one of the following:
\begin{itemize}
    \item a one-to-one correspondence between their open sets;
    \item a one-to-one correspondence between their closed sets; or
    \item a one-to-one correspondence between their complete systems of neighborhoods.
\end{itemize}

Both \emph{metric equivalence} and \emph{topological equivalence} define equivalence relations on the collection of all metric spaces.  
By Corollary~7.7, each equivalence class of metrically equivalent metric spaces is contained within an equivalence class of topologically equivalent ones.  
Determining to which topological equivalence class a metric space belongs is therefore a coarser—but also more fundamental—classification.  
According to Theorem~7.10, this classification is determined entirely by the family of open sets of the space, that is, by its \emph{topology}.

\newpage
\section{An Infinite-Dimensional Euclidean Space}

So far, we have studied finite-dimensional Euclidean spaces $\mathbf{R}^m$, consisting of all $m$-tuples
\[
\langle a_1, a_2, \ldots, a_m \rangle
\]
of real numbers. On this set, we define the function
\[
d(p, q) = \sqrt{(a_1 - b_1)^2 + \cdots + (a_m - b_m)^2} = \sqrt{\sum_{i=1}^m (a_i - b_i)^2},
\]
where $p = \langle a_1, \ldots, a_m \rangle$ and $q = \langle b_1, \ldots, b_m \rangle$. This function satisfies all the axioms of a metric and is called the \textbf{Euclidean metric} on $\mathbf{R}^m$. Unless otherwise stated, we always assume this metric on $\mathbf{R}^m$. The metric space $\mathbf{R}^m$ with the Euclidean metric is called the \textbf{Euclidean $m$-space}, also denoted by $E^m$.

\begin{theorem}
The Euclidean $m$-space $E^m$ is a metric space.
\end{theorem}

\begin{remark}
For $m=1$, the space $E^1$ coincides with the real line $\mathbf{R}$ with the standard metric. For $m=2$, the space $E^2$ coincides with the plane $\mathbf{R}^2$ with the usual Euclidean metric.
\end{remark}

Our next goal is to define a metric space $H$, sometimes called \textbf{Hilbert space}, which contains isometric copies of all finite-dimensional Euclidean spaces $(\mathbf{R}^n, d^n)$ as subspaces. A point $u \in H$ is a sequence
\[
u = (u_1, u_2, u_3, \ldots)
\]
of real numbers such that the series
\[
\sum_{i=1}^{\infty} u_i^2
\]
converges.

\begin{example}[Finite-dimensional embedding]
If $u = (u_1, \ldots, u_m, 0, 0, \ldots)$ is a sequence where only the first $m$ terms are nonzero, then $u$ can be identified with the point $(u_1, \ldots, u_m)$ in $E^m$. This shows that each finite-dimensional Euclidean space naturally embeds in $H$.
\end{example}

Given two points $u = (u_1, u_2, \ldots)$ and $v = (v_1, v_2, \ldots)$ in $H$, we intend to define a metric by
\[
d(u, v) = \left[ \sum_{i=1}^{\infty} (u_i - v_i)^2 \right]^{1/2}.
\]

\begin{remark}
Before this definition makes sense, we must ensure that the infinite series converges. Convergence will follow from the classical Cauchy-Schwarz inequality, which we state next.
\end{remark}

\begin{lemma}[Cauchy-Schwarz inequality, see \cite{Mendelson}]
Let $(u_1, u_2, \ldots, u_n)$ and $(v_1, v_2, \ldots, v_n)$ be $n$-tuples of real numbers. Then
\[
\sum_{i=1}^{n} u_i v_i \le \left( \sum_{i=1}^{n} u_i^2 \right)^{1/2} \left( \sum_{i=1}^{n} v_i^2 \right)^{1/2}.
\]
\end{lemma}

\begin{proof}
Consider, for an arbitrary real number $\lambda$, the nonnegative sum
\[
\sum_{i=1}^{n} (u_i + \lambda v_i)^2 \ge 0.
\]
Expanding, we get
\[
\sum_{i=1}^{n} u_i^2 + 2 \lambda \sum_{i=1}^{n} u_i v_i + \lambda^2 \sum_{i=1}^{n} v_i^2 \ge 0.
\]
Since this quadratic in $\lambda$ has a non-positive discriminant, it follows that
\[
\left( \sum_{i=1}^{n} u_i v_i \right)^2 \le \left( \sum_{i=1}^{n} u_i^2 \right) \left( \sum_{i=1}^{n} v_i^2 \right),
\]
as desired.
\end{proof}

\begin{corollary}
Let $u = (u_1, u_2, \ldots)$ and $v = (v_1, v_2, \ldots)$ belong to $H$, and define
\[
U = \sum_{i=1}^{\infty} u_i^2, \quad V = \sum_{i=1}^{\infty} v_i^2.
\]
Then the series $\sum_{i=1}^{\infty} u_i v_i$ converges absolutely, with
\[
\sum_{i=1}^{\infty} |u_i v_i| \le \sqrt{U} \, \sqrt{V}.
\]
\end{corollary}

\begin{proof}
For each positive integer $n$,
\[
\sum_{i=1}^{n} |u_i v_i| \le \left( \sum_{i=1}^{n} u_i^2 \right)^{1/2} \left( \sum_{i=1}^{n} v_i^2 \right)^{1/2} \le \sqrt{U} \, \sqrt{V}.
\]
Hence the partial sums of $\sum |u_i v_i|$ are bounded, implying that the series converges.
\end{proof}

Furthermore, let $\alpha, \beta \in \mathbf{R}$ and define
\[
\alpha u + \beta v = (\alpha u_1 + \beta v_1, \alpha u_2 + \beta v_2, \ldots).
\]
Then $\alpha u + \beta v \in H$, because the series
\[
\sum_{i=1}^{\infty} (\alpha u_i + \beta v_i)^2
\]
is absolutely convergent. Indeed, expanding the square gives three absolutely convergent series:
\[
\sum_{i=1}^{\infty} (\alpha u_i + \beta v_i)^2 = \alpha^2 \sum_{i=1}^{\infty} u_i^2 + 2 \alpha \beta \sum_{i=1}^{\infty} u_i v_i + \beta^2 \sum_{i=1}^{\infty} v_i^2.
\]

In particular, for $\alpha = \beta = 1$, we see that $u + v \in H$, and
\[
\sum_{i=1}^{\infty} (u_i + v_i)^2 = \sum_{i=1}^{\infty} u_i^2 + 2 \sum_{i=1}^{\infty} u_i v_i + \sum_{i=1}^{\infty} v_i^2.
\]
By Corollary 8.3, we have
\[
\sum_{i=1}^{\infty} (u_i + v_i)^2 \le \sum_{i=1}^{\infty} u_i^2 + 2 \left( \sum_{i=1}^{\infty} u_i^2 \right)^{1/2} \left( \sum_{i=1}^{\infty} v_i^2 \right)^{1/2} + \sum_{i=1}^{\infty} v_i^2 = (U^{1/2} + V^{1/2})^2.
\]

\begin{corollary}[Triangle Inequality in Hilbert Space, see \cite{Mendelson}]
\[
\left( \sum_{i=1}^{\infty} (u_i + v_i)^2 \right)^{1/2} \le U^{1/2} + V^{1/2}.
\]
\end{corollary}

\begin{remark}
This inequality generalizes the familiar triangle inequality from finite-dimensional Euclidean spaces to the infinite-dimensional Hilbert space $H$.
\end{remark}

\begin{theorem}[see \cite{Mendelson}]
The space $(H, d)$ is a metric space, where
\[
d(u, v) = \left( \sum_{i=1}^{\infty} (u_i - v_i)^2 \right)^{1/2}.
\]
\end{theorem}

\begin{proof}
The function $d$ clearly satisfies non-negativity, symmetry, and $d(u, v) = 0$ if and only if $u = v$.  

To verify the triangle inequality, let $a = (a_1, a_2, \ldots)$, $b = (b_1, b_2, \ldots)$, and $c = (c_1, c_2, \ldots)$ be elements of $H$. Set
\[
u = a - c, \quad v = c - b,
\]
so that $u_i = a_i - c_i$ and $v_i = c_i - b_i$. Then $u_i + v_i = a_i - b_i$, and applying Corollary 8.3 gives
\[
d(a, b) = \left( \sum_{i=1}^{\infty} (a_i - b_i)^2 \right)^{1/2} = \left( \sum_{i=1}^{\infty} (u_i + v_i)^2 \right)^{1/2} \le \left( \sum_{i=1}^{\infty} u_i^2 \right)^{1/2} + \left( \sum_{i=1}^{\infty} v_i^2 \right)^{1/2} = d(a, c) + d(c, b).
\]
Thus $d$ satisfies the triangle inequality, completing the proof that $(H, d)$ is a metric space.
\end{proof}

\begin{remark}
This theorem shows that $H$ is an infinite-dimensional generalization of Euclidean space. Every finite-dimensional Euclidean space $E^m$ can be viewed as a subspace of $H$, and the metric $d$ extends the standard Euclidean metric to sequences of infinite length.
\end{remark}

Let \( E^n \) be the collection of points \( u = (u_1, u_2, \ldots) \in H \) such that \( u_j = 0 \) for \( j > n \). To each point \( a = (a_1, a_2, \ldots, a_n) \in \mathbb{R}^n \) we can associate the point \( h(a) = (a_1, a_2, \ldots, a_n, 0, 0, \ldots) \in E^n \). Clearly \( h \) is a one-one mapping of \( \mathbb{R}^n \) onto the subspace \( E^n \) of \( H \). Using \( d'(a, b) = \left[ \sum_{i=1}^n (a_i - b_i)^2 \right]^{1/2} \) in \( \mathbb{R}^n \), \( d'(a, b) = d(h(a), h(b)) \). Since \( E^n \) is a metric space, \( (\mathbb{R}^n, d') \) is a metric space and \( h \) is an isometry of \( (\mathbb{R}^n, d') \) with \( (E^n, d|_{E^n}) \).

\begin{example}[Hilbert space isometric to a proper subspace, see \cite{Schaum}]
Let $H^*$ denote the proper subspace of $H$ consisting of all sequences whose first coordinate is zero:
\[
H^* = \{ \langle 0, a_2, a_3, \ldots \rangle \in H \}.
\] 
Consider the mapping 
\[
f : H \longrightarrow H^*, \quad f(\langle a_1, a_2, a_3, \ldots \rangle) = \langle 0, a_1, a_2, \ldots \rangle.
\] 
It is straightforward to check that $f$ is one-to-one, onto, and **distance-preserving**. Hence, the Hilbert space $H$ is **isometric** to a proper subspace of itself, illustrating the counterintuitive property of infinite-dimensional spaces that they can contain subspaces isometric to themselves.
\end{example}

\begin{definition}[Inner product space, see \cite{Chen}]
Let $E$ be a complex vector space. A mapping 
\[
\langle \cdot, \cdot \rangle : E \times E \longrightarrow \mathbb{C}
\] 
is called an \emph{inner product} on $E$ if for all $x, y, z \in E$ and $\alpha, \beta \in \mathbb{C}$ the following conditions hold:
\begin{enumerate}
    \item[(a)] Conjugate symmetry: $\langle x, y \rangle = \overline{\langle y, x \rangle}$;
    \item[(b)] Linearity in the first argument: $\langle \alpha x + \beta y, z \rangle = \alpha \langle x, z \rangle + \beta \langle y, z \rangle$;
    \item[(c)] Positive semi-definiteness: $\langle x, x \rangle \geq 0$;
    \item[(d)] Definiteness: $\langle x, x \rangle = 0$ implies $x = 0$.
\end{enumerate}
\end{definition}

\begin{example}[$N$-dimensional inner product space, see \cite{Chen}]
Consider $\mathbb{C}^{N}$, the space of all $N$-tuples $x = (x_1, \ldots, x_N)$ of complex numbers. Define the inner product by
\[
\langle x, y \rangle = \sum_{k=1}^{N} x_k \overline{y_k}.
\] 
This makes $\mathbb{C}^N$ a finite-dimensional inner product space.
\end{example}

\begin{example}[Infinite-dimensional inner product space, see \cite{Chen}]
Let $\ell^2$ denote the space of all sequences $x = (x_1, x_2, \ldots)$ of complex numbers such that
\[
\sum_{k=1}^{\infty} |x_k|^2 < \infty.
\] 
Define the inner product
\[
\langle x, y \rangle = \sum_{k=1}^{\infty} x_k \overline{y_k}.
\] 
This turns $\ell^2$ into an infinite-dimensional inner product space.
\end{example}

\begin{example}[Function space inner product, see \cite{Chen}]
Consider the space $\mathcal{C}([a,b])$ of all continuous complex-valued functions on $[a,b]$. Define
\[
\langle f, g \rangle = \int_a^b f(x) \overline{g(x)} \, dx.
\] 
Then $(\mathcal{C}([a,b]), \langle \cdot, \cdot \rangle)$ is an inner product space. This particular inner product will be frequently used in later constructions.
\end{example}

\begin{remark}
An inner product space is sometimes called a \emph{pre-Hilbert space}. Unlike ordinary finite-dimensional vector spaces, inner product spaces can be **infinite-dimensional**, such as spaces of sequences or functions. This property motivates the study of Hilbert spaces, which are complete inner product spaces.
\end{remark}

\begin{remark}
Given an inner product space $(E, \langle \cdot, \cdot \rangle)$, one can define a norm by
\[
\|x\| = \sqrt{\langle x, x \rangle}.
\] 
Before proving that this indeed defines a norm, it is necessary to establish **Schwarz's inequality** (Cauchy-Schwarz inequality) in inner product spaces.
\end{remark}

\begin{definition}[Norm in an inner product space]
Let $x$ be an element of an inner product space $H$. The quantity
\[
\|x\| = \sqrt{\langle x, x \rangle}
\]
is called the \textbf{norm} or \textbf{length} of $x$. Geometrically, it represents the distance of the point $x$ from the origin. This norm induces a natural metric on $H$ given by
\[
d(x, y) = \|x - y\|,
\]
so that $H$ becomes a metric space under the norm.
\end{definition}

\begin{example}[Norms in common spaces]
\leavevmode
\begin{enumerate}
    \item In $\mathbb{R}^n$ with the standard inner product 
    \(\langle x, y \rangle = \sum_{i=1}^{n} x_i y_i\), the norm is
    \[
    \|x\| = \sqrt{x_1^2 + x_2^2 + \cdots + x_n^2},
    \]
    corresponding to the usual Euclidean distance.

    \item In $\mathbb{C}^n$ with inner product 
    \(\langle x, y \rangle = \sum_{i=1}^{n} x_i \overline{y_i}\), the norm is
    \[
    \|x\| = \sqrt{\sum_{i=1}^{n} |x_i|^2}.
    \]

    \item In $L^2[a,b]$, the space of square-integrable functions, the inner product
    \(\langle f, g \rangle = \int_a^b f(x) \overline{g(x)} \, dx\) induces the norm
    \[
    \|f\| = \left( \int_a^b |f(x)|^2 \, dx \right)^{1/2}.
    \]
\end{enumerate}
\end{example}

\begin{theorem}[Cauchy–Schwarz inequality]
Let $H$ be an inner product space over $\mathbb{R}$ or $\mathbb{C}$. For all $x, y \in H$,
\[
|\langle x, y \rangle| \leq \|x\| \, \|y\|.
\] 
Equality holds if and only if $x$ and $y$ are linearly dependent.
\end{theorem}

\begin{proof}
If $y = 0$, the inequality is immediate. Otherwise, for any scalar $\lambda$,
\[
0 \leq \|x - \lambda y\|^2 = \langle x - \lambda y, x - \lambda y \rangle = \|x\|^2 - 2 \operatorname{Re}(\lambda \langle x, y \rangle) + |\lambda|^2 \|y\|^2.
\]
Choosing \(\lambda = \frac{\langle x, y \rangle}{\|y\|^2}\) minimizes the right-hand side, giving
\[
0 \leq \|x\|^2 - \frac{|\langle x, y \rangle|^2}{\|y\|^2}.
\]
Multiplying both sides by $\|y\|^2$ yields
\[
|\langle x, y \rangle|^2 \leq \|x\|^2 \|y\|^2,
\]
and taking square roots completes the proof.
\end{proof}

\begin{example}[Cauchy–Schwarz in $\mathbb{R}^n$]
For $x, y \in \mathbb{R}^n$ with the standard inner product,
\[
\left| \sum_{i=1}^{n} x_i y_i \right| \leq 
\sqrt{\sum_{i=1}^{n} x_i^2} \, \sqrt{\sum_{i=1}^{n} y_i^2},
\]
which is the classical vector Cauchy–Schwarz inequality.
\end{example}

\begin{theorem}[Norm induced by inner product, see \cite{Chen}]
Every inner product space $(H, \langle \cdot, \cdot \rangle)$ is also a normed vector space with norm
\[
\|x\| = \sqrt{\langle x, x \rangle}.
\]
\end{theorem}

\begin{proof}
We verify the norm axioms:
\begin{enumerate}
    \item \(\|x\| = 0 \iff x = 0\) (from definiteness of the inner product);
    \item \(\|\alpha x\| = |\alpha| \, \|x\|\) for all scalars $\alpha$:
    \[
    \|\alpha x\| = \sqrt{\langle \alpha x, \alpha x \rangle} = \sqrt{\alpha \overline{\alpha} \langle x, x \rangle} = |\alpha| \, \|x\|;
    \]
    \item (Triangle inequality) For $x, y \in H$:
    \[
    \begin{aligned}
    \|x + y\|^2 &= \langle x + y, x + y \rangle = \|x\|^2 + 2 \operatorname{Re} \langle x, y \rangle + \|y\|^2 \\
    &\leq \|x\|^2 + 2|\langle x, y \rangle| + \|y\|^2 \\
    &\leq \|x\|^2 + 2\|x\| \|y\| + \|y\|^2 \quad \text{(by Cauchy–Schwarz)} \\
    &= (\|x\| + \|y\|)^2.
    \end{aligned}
    \]
\end{enumerate}
Taking square roots gives the triangle inequality \(\|x + y\| \leq \|x\| + \|y\|\).
\end{proof}

\begin{theorem}[Parallelogram Law, see \cite{Chen}]
Let $x$ and $y$ be elements of an inner product space $H$. Then the following identity holds:
\[
\|x + y\|^{2} + \|x - y\|^{2} = 2(\|x\|^{2} + \|y\|^{2}).
\]
\end{theorem}

\begin{remark}
The Parallelogram Law illustrates a fundamental geometric property of inner product spaces: the sum of the squares of the diagonals of a parallelogram equals the sum of the squares of all sides. This identity is crucial in characterizing norms that are induced by an inner product.
\end{remark}

\begin{theorem}[Pythagorean Theorem in Inner Product Spaces, see \cite{Chen}]
Let $x$ and $y$ be orthogonal vectors in an inner product space $H$, i.e., $\langle x, y \rangle = 0$. Then
\[
\|x + y\|^{2} = \|x\|^{2} + \|y\|^{2}.
\]
\end{theorem}

\begin{remark}
This theorem generalizes the familiar Pythagorean theorem from Euclidean geometry to any inner product space. Orthogonality replaces the role of perpendicularity, and the squared norm corresponds to the squared length.
\end{remark}

\begin{example}
In $\mathbb{R}^2$ with the standard inner product, if $x = (x_1, x_2)$ and $y = (y_1, y_2)$ are perpendicular vectors, then
\[
\|x + y\|^2 = x_1^2 + x_2^2 + y_1^2 + y_2^2 = \|x\|^2 + \|y\|^2,
\]
recovering the classical Pythagorean theorem.
\end{example}

In this chapter, we have extended the concept of Euclidean space to infinite dimensions, introducing the Hilbert space $H$ as a fundamental example. We explored its metric structure, demonstrated that it contains isometric copies of finite-dimensional Euclidean spaces, and examined the rich algebraic and geometric properties arising from its inner product. Key results, including the Triangle Inequality, Cauchy–Schwarz inequality, the Parallelogram Law, and the Pythagorean theorem, illustrate how notions of length, distance, and orthogonality generalize to infinite-dimensional settings. 

These foundational concepts establish a framework for further studies in functional analysis, operator theory, and applications in physics and engineering, where infinite-dimensional spaces naturally arise. By understanding both the metric and inner product structures, one gains the tools necessary to rigorously analyze convergence, orthogonality, and projection in more advanced contexts.

\newpage
\section*{Acknoledgements}
We would like to thank Prof. Daniel Petriceanu and Prof. Andrei Alexandrescu for their valuable advice and guidance throughout the writing process of this paper, as well as their paramount importance in the scene of Romanian high school mathematics. We would also like to thank all of the authors of the textbooks, articles, and journals that led up to the creation of our paper, as well as the publishing houses that made the aforementioned sources of information readily available.

\end{document}